\documentclass[12pt, twoside]{article}
\usepackage[dvips]{graphicx}
\usepackage{amssymb,color,latexsym}
\usepackage[centertags]{amsmath}
\usepackage{amsfonts}
\usepackage{comment}
\usepackage{amsmath}
\usepackage{enumerate}
\usepackage{rotating}
\usepackage{multirow}

\def\titlerunning#1{\gdef\titrun{#1}}
\makeatletter
\def\author#1{\gdef\autrun{\def\and{\unskip, }#1}\gdef\@author{#1}}
\def\address#1{{\def\and{\\\hspace*{18pt}}\renewcommand{\thefootnote}{}%
\footnote {#1}}%
\markboth{\autrun}{\titrun}}
\makeatother
\def\email#1{e-mail: #1}
\def\subjclass#1{{\renewcommand{\thefootnote}{}%
\footnote{\emph{Mathematics Subject Classification (2010):} #1}}}
\def\keywords#1{\par\medskip
\noindent\textbf{Keywords.} #1}

\newtheorem{theorem}{Theorem}[section]
\newtheorem{lemma}[theorem]{Lemma}
\newtheorem{corollary}[theorem]{Corollary}
\newtheorem{proposition}[theorem]{Proposition}

\newcommand{\fl}{\mathcal{F}}

\newcommand{\Aut}{\mathrm{Aut}}
\newcommand{\gr}{\mathcal{G_M}}
\newcommand{\Mon}{\mathrm{Mon}}
\newcommand{\oo}{\mathcal{O}}
\newcommand{\Orb}{\mathrm{Orb}}
\newcommand{\ma}{\mathcal{M}}

\begin{document}


\baselineskip=17pt


\titlerunning{Medial symmetry type graphs}

\title{Medial symmetry type graphs}

\author{Isabel Hubard
  \and Alen Orbani\'c
  \and Toma\v{z} Pisanski
\and Mar\'ia del R\'io Francos
}


\maketitle

\address{I. Hubard: Instituto de Matem\'{a}ticas, Universidad Nacional Aut\'{o}noma de M\'{e}xico, M\'{e}xico;\\
               \email{hubard@matem.unam.mx }
\and
                A. Orbani\'c: Faculty of Mathematics, Physics and Mechanics, University of Ljubljana, Slovenia; 
               \email{alen.orbanic@fmf.uni-lj.si}
\and        
               T. Pisanski: Faculty of Mathematics, Physics and Mechanics, University of Ljubljana, Slovenia; 
               \email{tomaz.pisanki@fmf.uni-lj.si}
\and
               M. del R\'io Francos: Institute of Mathematics Physics and Mechanics, University of Ljubljana, Slovenia;
	    \email{maria.delrio@fmf.uni-lj.si}
}



\begin{abstract}

A $k$-orbit map is a map with its automorphism group partitioning the set of flags into $k$ orbits. Recently $k$-orbit maps were studied by Orbani\' c, Pellicer and Weiss, for $k \leq 4$. In this paper we use symmetry type graphs to extend such study and classify all the types of $5$-orbit maps, as well as all self-dual, properly and improperly, symmetry type of $k$-orbit maps with $k\leq 7$. Moreover, we  determine, for small values of $k$, all types of $k$-orbits maps that are medial maps.  Self-dualities constitute an important tool in this quest.

\end{abstract}

\section{Introduction}

Exploring symmetry and its boundaries has been a driving force of the progress of mathematics already in ancient times with the Platonic and Archimedean solids being a prime example. Regular maps and polytopes represent a modern \\
generalization of the Platonic solids. They have a common feature, namely, that the group of automorphisms acts regularly on the set of elements, usually, called flags, that constitute the object under investigation.

Maps on closed surfaces may be completely described by trivalent edge-colored graph, known as the \emph{flag graph}.
This is equivalent to the description proposed by Lins in 1982 \cite{Lins}. For a modern treatment of the subject compare \cite{AGT}.
The \emph{symmetry type graph} of a map is a trivalent edge-colored factor of its flag graph obtained from the action of the group of automorphisms of the map on the flags. This notion is equivalent to the Delaney-Dress symbol described in \cite{DressHuson87}.  An application to mathematical chemistry is given in \cite{DrBr96} and a strategy of how to generate them is shown in \cite{CompSymTypeGraph}. In the symmetry type graph, as it name says, one can find enough information in regard to symmetries of certain type of map. This graph is of our interest since it lets us know if the map is regular, transitive on either vertices, edges or faces, among other properties that will be studied through this paper.

The \emph{medial }of a map is a map that arises from a similar operation to the truncation on a map \cite{FowlerPisanski,PisanskiRandic,PiZi}. This is one of the Wythoffian constructions that can be performed on an arbitrary map on a surface.
By Wythoffian construction, we refer in particular to drop a perpendicular to the hypotenuse in a right-angle triangle, from the vertex in the right-angle; in our case this triangle is a flag of the original map. The medial map is called ``1-ambo" by John Conway \cite{Conway}. Some of these Wythoffian operations have been used in different contexts \cite{FowlerPisanski, PisanskiRandic} and were described by Pisanski and \v{Z}itnik in a recent chapter 8 of \cite{PiZi}.
The medial operation can be described as subdivision of flag triangles (see for example \cite{MonGp_Self-Inv,PiZi}). Equivalently, we may describe this operation as rules transforming the flag graph of the original map to the flag graph of its medial, and work with its symmetry type graph. 

This paper is organised as follows.
Section~\ref{sec:maps} gives an introduction to maps, the dual and Petrie-dual maps of a given map as well as to the medial operation.
In Section~\ref{sec:STG} we develop the concept of the symmetry type graph of a map and enumerate all symmetry type graphs of maps with at most 5 flag orbits.
Moreover, we analyse how the dualities of a map work on its symmetry type graph to define the extended symmetry type graph of a self-dual map. We further enumerate all extended symmetry type graphs of self-dual maps with at most 7 flag orbits. The reader is referred to \cite{OperationsMaps} and \cite{Self-DualSelf-Petrie} for further details.
Section~\ref{sec:med} deals with how to obtain the symmetry type graph of the medial of a map, by operations on the (extended) symmetry type graph of the map. We enumerate all the medial symmetry type graphs with at most 7 vertices.
 In particular we show that every type of edge-transitive map is a medial type.

\section{Maps}
\label{sec:maps}

We start this section by giving the basic theory of maps, their flag graphs and monodromy groups. We also review the concepts of duality and Petrie-duality; we introduce properly and improperly self-dual maps, as well as the medial operation on maps.

We shall say that a {\em map} $\ma$ is a $2$-cellular embedding of a connected graph $G$ on a compact surface without boundary, in the sense that the graph separates the surface into simply connected regions.
The vertices and edges of the map are the same as those of its underlying graph, and the faces of $\ma$ are described by some {\em distinguished closed walks} of $G$, in such a way that each edge of $G$ is in either exactly two distinguished cycles, or twice on the same one. (Note that the distinguished cycles of $G$ can be identified with the simply connected regions obtained by removing the graph from the surface.)
The set of vertices, edges and faces of $\ma$ will be denoted by $V$, $E$ and $F$, respectively.
And, for convenience, we shall often refer to them as the $0$-, $1$- and $2$-faces of $\ma$, respectively.
A \emph{flag} of a map is defined by an ordered triple $\{v,e,f\}$ of mutually incident vertex, edge and face of the map. The set of all flags of $\ma$ will be denoted by $\fl(\ma)$. We shall say that a map $\ma$ is {\em equivelar} if all its faces have the same number of edges (say $p$), and its vertices have the same valency (say $q$). In this case we say that $\ma$ has {\em Sch\"afli type} $\{p,q\}$.

By selecting one point in the interior of each $i$-face of $\ma$, one can identify a flag $\Phi :=\{v,e,f\}$ with the triangle with vertices $v$, and the chosen interior points of $e$ and $f$.
 By doing this with every flag of $\ma$, we obtain the {\em barycentric subdivision} $BS(\ma)$ of a map $\ma$. 
Hence, $BS(\ma)$ is a triangular map on the same surface, where its vertices correspond either to a vertex, the midpoint of an edge, or the chosen point in the interior of a face of $\ma$.
This gives a natural colouring on the vertex set of the barycentric subdivision, with colours $0, 1, 2$, depending on whether the vertex of $BS(\ma)$ corresponds to a vertex,  an edge or a face of $\ma$.
Hence, each 2-face of the map $\ma$ is decomposed into triangles of $BS(\ma)$, all of them having the same vertex of colour 2.

Note that each triangle $\Phi$ of $BS(\ma)$ shares an edge with exactly three other triangles that we shall denote by  $\Phi^0, \Phi^1$ and $\Phi^2$, where the triangles $\Phi$ and $\Phi^i$ share the vertices of colours $j$ and $k$, but differ on that of colour $i$. These triangles correspond to {\em adjacent} flags of $\ma$, or more specifically, to {\em $i$-adjacent} flags, if their corresponding triangles in $BS(\ma)$ differ exactly on the vertex of colour $i$.
We extend this notation by induction in the following way, $(\Phi^{i_0,i_1,\dots, i_{k-1}})^{i_k} = \Phi^{i_0,i_1,\dots, i_{k}}$ and note that $\Phi^{i,i}=\Phi$ for  $i=0,1,2$ and $\Phi^{0,2} = \Phi^{2,0}$, for every flag $\Phi$.
Moreover, the connectivity of the underlying graph of $\ma$ implies that given any two flags $\Phi$ and $\Psi$ of $\ma$, there exist integers $i_0,i_1,\dots, i_{k} \in \{0, 1, 2\}$ such that $\Psi=\Phi^{i_0,i_1,\dots, i_{k}}$.

If for each flag of $\ma$ (faces of $BS(\ma)$) we assign a vertex, and define an edge between two of them whenever the corresponding flags are adjacent, we obtain a new graph.
We can naturally colour the edges of this graph with colours $0, 1, 2$ in such a way that the edge between any two $i$-adjacent flags has colour $i$.
The obtained $3$-edge-coloured graph is called the {\em flag graph} $\gr$ of $\ma$. (Note that this trivalent graph with the vertex set $\fl(\ma)$ defines the dual of the barycentric subdivision $BS(\ma)$, in the sense of Section \ref{dual}.)

To each map $\ma$ we can associate a subgroup of the permutation group of the flags of $\ma$ in the following way. Let $s_0, s_1, s_2$ be permutations of the set $\fl(\ma)$, acting on the right, such that for any flag $\Phi \in \fl(\ma)$,
                               $$\Phi \cdot s_i = \Phi^i;$$
for each $i = 0, 1, 2$. It is straightforward to see that these permutations generate a group $\Mon(\ma)$,  called the {\em monodromy (or connection) group} of the map $\ma$ \cite{MonGp_Self-Inv},
and satisfy the following properties.
\begin{itemize}
\item[(i)] $s_0$, $s_1$, and $s_2$ are fixed-point free involutions;
\item[(ii)] $s_0s_2 = s_2s_0$, and $s_0s_2$ is fixed-point free;
\item[(iii)] The group $\Mon(\ma)$ is transitive on $\fl(\ma)$.
\end{itemize}

One can then use the {\em distinguished generators} $s_0, s_1, s_2$ of $\Mon(\ma)$ to label the edges of the flag graph $\gr$ in a natural way, that is, each edge of colour $i$ is labelled with the generator $s_i$.
In this way, one can think of the walks among the edges of $\gr$ as words in $\Mon(\ma)$.
In fact, if $\Phi, \Psi \in \fl(\ma)$ are two flags such that $\Psi = \Phi^w$, for some $w=s_{i_0}s_{i_1} \dots s_{i_k} \in \Mon(\ma)$, then the walk of $\gr$ starting at the vertex $\Phi$ and traveling in order among the edges $i_0, i_1, \dots, i_k$ will finish at the vertex $\Psi$.
And vice versa, every walk among the coloured edges of $\gr$ starting at $\Phi$ and finishing at $\Psi$ induces a word $w\in\Mon(\ma)$ that satisfies that $\Phi^w=\Psi$.
Note however that in general the action of $\Mon(\ma)$ is not semiregular on $\fl(\ma)$, implying that one can have different ``coloured'' walks in $\gr$ going from $\Phi$ to $\Psi$ that induce different words of $\Mon(\ma)$ that act on the flag $\Phi$ in the same way.

In $\gr$, the edges of a given colour form a perfect matching (an independent set of edges containing all the vertices of the graph).
Hence the union of two sets of edges of different colour is a subgraph whose components are even cycles. Such subgraph is called a {\em 2-factor} of $\gr$.
In particular, note that since $(s_0s_2)^2 = 1$ and $s_0s_2$ is fixed-point free, the cycles with edges of alternating colours 0 and 2 are all of length four and
these 4-cycles define the set of edges on the map.
In other words, the edges of $\ma$ can be identified with the orbits of $\fl(\ma)$ under the action of the subgroup generated by the involutions $s_0$ and $s_2$; that is, $E(\ma)= \{ \Phi^{\langle s_0,s_2 \rangle} \mid \Phi \in \fl(\ma) \}$.

Similarly, we find that the vertices and faces of $\ma$ are identified with the respective orbits of the subgroups $\langle s_1,s_2 \rangle$ and $\langle s_0,s_1 \rangle$ on $\fl(\ma)$. That is, $V(\ma)= \{ \Phi^{\langle s_1,s_2 \rangle}  \mid \Phi \in \fl(\ma)  \}$ and $F(\ma)= \{ \Phi^{\langle s_0,s_1 \rangle}  \mid \Phi \in \fl(\ma) \}$.
Thus, the group $\langle s_0, s_1, s_2 \rangle$ acts on set of all $i$-faces of $\ma$ transitively, for each $i \in \{0,1,2\}$.

In particular, for each flag $\Phi \in \fl(\ma)$, the set
$(\Phi)_0:=\{\Phi^w \mid w \in \langle s_1,s_2 \rangle\}$ is the orbit of the flag $\Phi$ around a vertex of $\ma$. Similarly, $(\Phi)_1:=\{\Phi^w \mid w \in \langle s_0,s_2 \rangle\}$ and $(\Phi)_2:=\{\Phi^w \mid w \in \langle s_0,s_1 \rangle\}$ are the orbits of the flag $\Phi$ around an edge and a face of the map $\ma$. The following lemma states that in fact, for each $k$, $(\Phi)_k$ is precisely the set of flags containing the $k$-face of $\Phi$, and hence we can identify the $k$-face of $\Phi$ with the set $(\Phi)_k$.

\begin{lemma}
\label{k-face}
Let $\ma$ be a map, $\fl(\ma)$ its set of all flags and $\Mon(\ma)=\langle s_0, s_1, s_2 \rangle$ be the monodromy group of $\ma$. For each $\Phi \in \fl(\ma)$ and $k \in \{0,1,2\}$, let $(\Phi)_k:= \{ \Phi^w \mid w \in \langle s_i, s_j \rangle, \ i,j \neq k \}$.
  If $\Phi, \Psi \in \fl(\ma)$
and $k \in \{0,1,2\}$ are such that $(\Phi)_k \cap (\Psi)_k \neq \emptyset$, then $(\Phi)_k=(\Psi)_k$.
\end{lemma}

\begin{proof}
If $(\Phi)_k \cap (\Psi)_k \neq \emptyset$, then there exist $w_0, w_1 \in \langle s_i, s_j\rangle$ such that $\Phi^{w_0}=\Psi^{w_1}$. But for any $w \in \langle s_i,s_j\rangle$ we have that $\Phi^w = \Phi^{w_0w_0^{-1}w}=\Psi^{w_1w_0^{-1}w}$, and since $w_1w_0^{-1}w \in \langle s_i,s_j\rangle$, then $\Phi^w \in (\Psi)_k$, implying that $(\Phi)_k \subseteq (\Psi)_k$. A similar argument shows the other contention.
\end{proof}

Throughout the paper we treat the $k$-face of a flag $\Phi$ and the corresponding $(\Phi)_k$ as the same thing.

An {\em automorphism} of $\gr$ is a bijection of the vertices of $\gr$ that preserves the incidences of the graph; the set of all automorphisms of $\gr$ is the {\em automorphism group} $\Aut(\gr)$ of $\gr$. There are two interesting subgroups of $\Aut(\gr)$ associated with the graph $\gr$: the {\em colour respecting automorphism group} $\Aut_r(\gr)$,  consisting of all automorphisms of $\gr$ that induce a permutation of the colours of the edges,
and the {\em  colour preserving automorphism group} $\Aut_c(\gr)$, consisting of all automorphisms of $\gr$ that send two adjacent vertices by colour $i$ into other two adjacent by the same colour $i$. Clearly $\Aut_c(\gr) \leq \Aut_r(\gr) \leq \Aut(\gr)$.

A bijection $\gamma$ of the vertices, edges and faces of the map $\ma$ which preserves the map is called an {\em automorphism} of  $\ma$. We shall denote by $\Gamma(\ma)$ the group of automorphisms of $\ma$. Note that every automorphism of $\ma$ induces a bijection on the set of flags of $\ma$ that preserves adjacencies.
Thus, $\Gamma(\ma)$ can be seen as the subgroup of $Sym(\fl(\ma))$ that preserves the (coloured) adjacencies.
In other words, an automorphism of a map $\ma$ is an edge-colour preserving automorphism of the flag graph $\gr$; that is, $\Gamma(\ma) = \Aut_c(\gr)$. Moreover, a bijection $\gamma$ of $\fl(\ma)$ is an automorphism of the map $\ma$ if and only if it ``commutes'' with the distinguished generators $s_0, s_1, s_2$ of $\Mon(\ma)$. That is, for every $i =0,1,2$ and every $\Phi \in \fl(\ma)$, 
$\Phi^{s_i}\gamma=(\Phi\gamma)^{s_i}$.

Let $\gamma \in \Aut_c(\gr)$ be such that $\Phi \gamma = \Phi$ for some vertex $\Phi$ of $\gr$. Since $\gamma$ preserves the colours of the edges of $\gr$, it must fix all the edges incident to $\Phi$ and hence all neighbours of $\Phi$. It is not difficult to see that, by connectivity, $\gamma$ fixes all the vertices of $\gr$, as well as all its edges, that is, $\gamma$ is the identity element of $\Aut_c(\gr)$. This implies that the action of $\Aut_c(\gr)$ on the vertices of $\gr$ is semiregular and therefore the action of $\Gamma(\ma)$ is semiregular on the flags of $\ma$.
Hence,  all the orbits on flags under the action of $\Gamma(\ma)$ have the same size.

We say that the map $\ma$ is a {\em $k$-orbit map} whenever the automorphism group $\Gamma(\ma)$ has exactly $k$ orbits on $\fl(\ma)$. Furthermore, if $k=1$ (i.e. the action of $\Gamma(\ma)$ is  transitive on the flags), then we say that $\ma$ is a {\em regular} map.
A {\em chiral} map is a 2-orbit map such that every flag and its adjacent ones are  in different orbits.

\subsection{Dual and self-dual maps}
\label{dual}

A {\em duality} $\delta$ from a map $\ma$ to a map $\mathcal{N}$ is a bijection from the set of flags $\fl(\ma)$ of $\ma$ to the set of flags $\fl(\mathcal{N})$ of $\mathcal{N}$ such that for each flag $\Phi \in \fl(\ma)$ and each $i \in \{0,1,2\}$, $\Phi^i\delta = (\Phi\delta)^{2-i}$. If there exists a duality from $\ma$ to $\mathcal{N}$, we say that $\mathcal{N}$ is the {\em dual} map of $\ma$, and we shall denote it by $\ma^*$. Note that $(\ma^*)^*\cong \ma$.
In terms of the flag graphs, a duality can be regarded as a bijection between the vertices of $\gr$ and the vertices of $\mathcal{G_{\ma^*}}$ that sends edges of colour $i$ of $\gr$ to edges of colour $2-i$ of $\mathcal{G_{\ma^*}}$, for each $i \in \{0,1,2\}$.

If there exists a duality from a map $\ma$ to itself, we shall say that $\ma$ is a {\em self-dual map}. Given $\delta, \omega$ two dualities of a self-dual map $\ma$, and $\Phi \in \fl(\ma)$, we have that $\Phi^i\delta\omega=(\Phi\delta)^{2-i}\omega= (\Phi\delta\omega)^i$, implying that $\delta\omega$ is an automorphism of $\ma$.
Thus, the product of two dualities of a self-dual map is no longer a duality, but an automorphism of the map. In particular, the square of any duality is an automorphism. The set of all dualities and automorphisms of a map $\ma$ is called the {\em extended group} $\mathcal{D}(\ma)$ of the map $\ma$. The automorphism group $\Gamma(\ma)$ is then a subgroup of index at most two in $\mathcal{D}(\ma)$. In fact, the index is two if and only if the map is self-dual.

For each flag $\Phi \in \fl(\ma)$ denote by $\oo_\Phi$ the orbit of $\Phi$ under the action of ${\Gamma(\ma)}$, and denote by $\mathrm{Orb}(\ma) := \{\oo_\Phi \mid \Phi \in \fl(\ma)\}$ the set of all the orbits of $\fl(\ma)$ under $\Gamma(\ma)$. Hubard and Weiss showed the following very useful lemma in \cite{isasia}.

\begin{lemma}
\label{dualityaction}
Let $\ma$ be a self-dual map, $\delta$ a duality of $\ma$ and $\oo_1, \oo_2 \in \Orb(\ma)$.
If $\delta$ sends a flag from $\oo_1$ to a flag in $\oo_2$, then all the dualities send flags of $\oo_1$ to flags in $\oo_2$.
\end{lemma}

This lemma allows us to divide the self-dual maps into two different classes.
Given $\ma$ a self-dual map, we say that $\ma$ is {\em properly self-dual} if its dualities preserve all flag-orbits of $\ma$.
Otherwise, we say that $\ma$ is {\em improperly self-dual}.

One can take a more algebraic approach in dealing with dual maps and dualities.
In fact, if $\ma$ is a map with monodromy group $\Mon(\ma)$ generated by $s_0, s_1, s_2$, the monodromy group of the dual map $\ma^*$ is generated by $s_2, s_1, s_0$.
Moreover, in \cite{MonGp_Self-Inv} was proved that a map $\ma$ is self-dual if and only if $d:\Mon(\ma) \to \Mon(\ma)$ sending $s_i$ to $s_{2-i}$ is a group automorphism such that $d(N)$ and $N$ are conjugated, where $N = Stab_{\Mon(\ma)}(\Phi)$ and $\Phi \in \fl(\ma)$.
In other words, this latter implies that $s_0 = d^{-1}s_2d$, $s_1 = d^{-1}s_1d$ and $s_2 = d^{-1}s_0d$.

\subsection{Petrie-dual map}

In \cite{Coxeter}, Coxeter introduced the Petrie-polygons and extended this concept to any dimension.
According to Coxeter, it was John Flinders Petrie who proposed the use of the ``zig-zag" polygons, in manner to find more regular polyhedra (with an infinite number of faces).
A {\em Petri polygon} is a ``zig-zag" path among the edges of a map $\ma$ in which every two consecutive edges, but not three, belong to the same face.
Note that each edge of a Petrie polygon appears either just once in exactly two different Petrie polygons of $\ma$, or twice in the same Petrie polygon of $\ma$.
Hence we can define a map with the same set of vertices and edges of $\ma$, but with the Petrie polygons as faces. This map is known as the {\em Petrie-dual} (or {\em Petrial}) map of $\ma$, and is denoted by $\ma^P$. If a map $\ma$ and its Petrie-dual $\ma^P$ are isomorphic maps, then $\ma$ is said to be {\em self-Petrie}.

Let $s_0, s_1, s_2$ be the distinguished generators of $\Mon(\ma)$.
Since the set of vertices and edges in $\ma^P$ coincide with those of $\ma$, then the set of flags of $\ma^P$ coincide with $\fl(\ma)$.
Even more, two flags in the flag graph $\mathcal{G}_{\ma^P}$ are 1-adjacent if and only if they are adjacent in the flag graph $\gr$ by $s_1$ (by the definition of the Petrie polygon).
However, recall that a walk along the $s_0$ and $s_1$ edges of $\gr$ define a face of $\ma$, but a face in $\ma^P$ corresponds to a \textquotedblleft zig-zag" path in $\ma$.
Hence, two flags $\Phi,\Psi\in\fl(\ma^P)$ are 0-adjacent in $\mathcal{G}_{\ma^P}$ if and only if $\Phi^{0,2} = \Psi$ in $\gr$.
Thus, the set of faces of $\ma^P$ is defined as $\{ \Phi^{\langle s_0s_2,s_1 \rangle} \mid \Phi \in \fl(\ma) \}$.
The 4-cycles that represent the edges of $\ma$ are no longer cycles of the flag-graph $\mathcal{G}_{\ma^P}$.
However, since a flag and its 2-adjacent flag in $\ma^P$ differ only on the face, and the vertices and edges of $\ma^P$ are the same as those of $\ma$, $\Phi$ and $\Psi$ are 2-adjacent in $\ma^P$ if and only if they are 2-adjacent in $\ma$.

Therefore, there is a bijection, $\pi$ say, between the vertices of $\gr$ and the vertices of $\mathcal{G}_{\ma^P}$ that preserves the colours 1 and 2, and interchanges each (0,2)-path by an edge of colour 0. The monodromy group of the Petrie dual map of $\ma$ is generated by $s_0s_2, s_1, s_2$ (where, as before, $s_0, s_1, s_2$ are the generators of $\Mon(\ma)$). Wilson and Lins, in \cite{Wilson} and \cite{Lins}, respectively, showed that for a map $\ma$, it can be seen that the bijection $\pi$ and the duality $\delta$ are operators on $\ma$ that generate a subgroup of $Sym(\fl(\ma))$ isomorphic to $S_3$; where $\delta \circ \pi \circ \delta = \pi \circ \delta \circ \pi$ is the third element of order two in it, which defines a bijection between the set of flags of $\ma$ and the set of flags of a map $\mathcal{N}$ known as the {\em opposite} map of $\ma$.

\subsection{The medial operation}

There is an interesting operation on maps called {\em the medial} of a map (see \cite{CRP,PisanskiRandic}). For any map $\ma$, we define the medial of $\ma$, $Me(\ma)$, in the following way.
The vertex set of $Me(\ma)$ is the edge set of $\ma$, two vertices of $Me(\ma)$ have an edge joining them if the corresponding edges of $\ma$ share a vertex and belong to the same face. This gives raise to a graph embedded on the same surface than $\ma$. Hence, the faces of $Me(\ma)$ are simply the connected regions of the complement of the graph on the surface. It is then not difficult to see that the face set of $Me(\ma)$ is in one to one correspondence with the set containing all faces and vertices of $\ma$.
Hence, it is straightforward to see that the medial of a map $\ma$ and the medial of its dual $\ma^*$ are isomorphic.

We note that every flag of the original map $\ma$ is divided into two flags of the medial $Me(\ma)$. In fact, given a flag $\Phi=\{v,e,f\}$ of $\ma$, one can write the two flags of $Me(\ma)$ corresponding to $\Phi$ as $(\Phi, 0):=\{e,\{v,f\},v\}$ and $(\Phi,2):=\{e, \{v,f\},f\}$.
It is then straightforward to see that the adjacencies of the flags of $Me(\ma)$ are closely related to those of the flags of $\ma$. In fact, we have that, if $S_0, S_1$ and $S_2$ are the distinguished generators of $\Mon(Me(\ma))$, then,
\begin{eqnarray*}
(\Phi,0)\cdot S_0=(\Phi\cdot s_1,0), \;\; (\Phi,0)\cdot S_1=(\Phi\cdot s_2,0), \;\; (\Phi,0)\cdot S_2=(\Phi,2), \\
(\Phi,2)\cdot S_0=(\Phi\cdot s_1,2), \;\; (\Phi,2)\cdot S_1=(\Phi\cdot s_0,2), \;\; (\Phi,2)\cdot S_2=(\Phi,0).
\end{eqnarray*}
Moreover, the valency of every vertex of a medial map $Me(\ma)$ is 4 and if the original map $\ma$ is equivelar of Schl\"afli type $\{p,q\}$ then the faces of $Me(\ma)$ are $p$-gons and $q$-gons. Therefore $Me(\ma)$ is equivelar if and only if $p=q$; in such case $Me(\ma)$ has Schl\"afli type $\{p,4\}$.

It is now easy to obtain the flag graph of $Me(\ma)$ from the flag graph of $\ma$. An algorithm showing how to do this is indicated in Figure~\ref{med-flag}.

  \begin{figure}[htbp]
      \begin{center}
\vspace{-9cm}
\includegraphics[width=12.5cm]{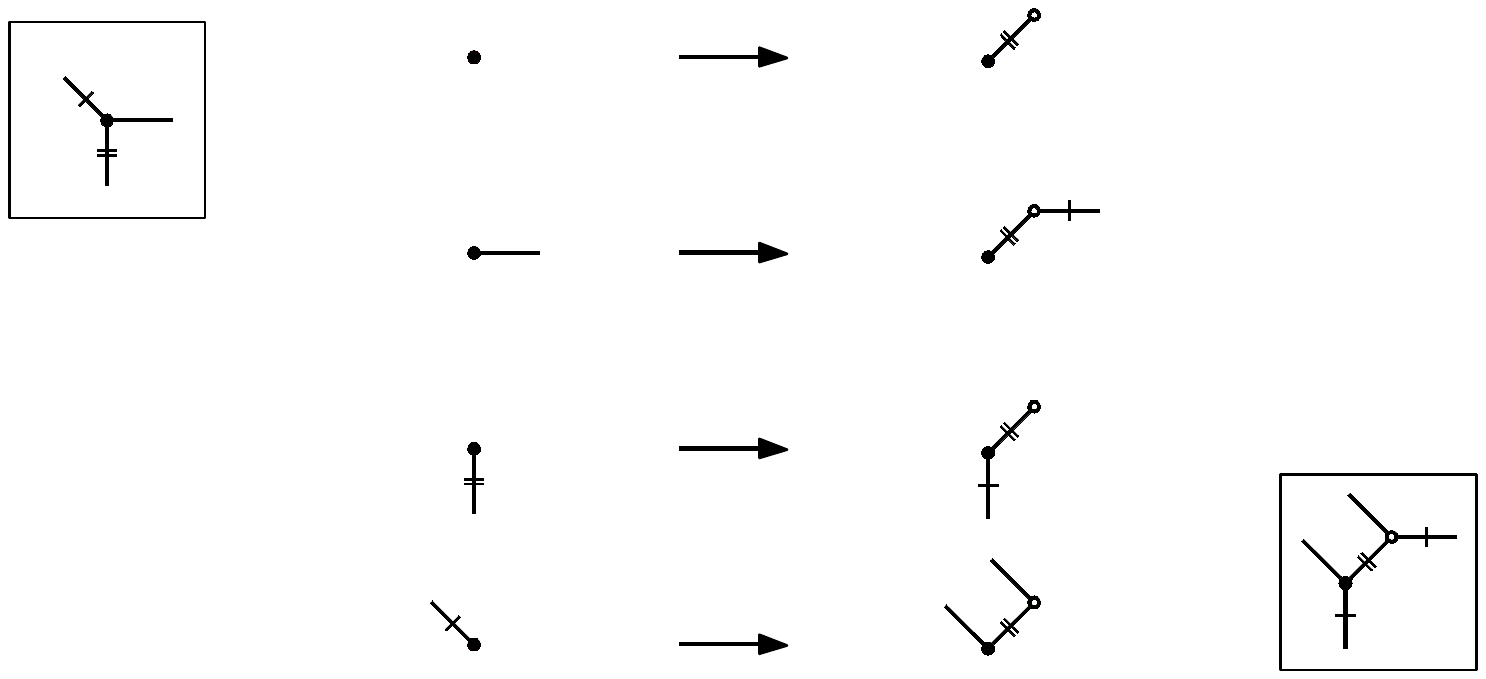}
\vspace{-4cm}
\caption{Any local representation of a flag, in the left. The result under the medial operation, locally obtained, in the right.}\label{med-flag}
      \end{center}
  \end{figure}

In \cite{MonGp_Self-Inv}, Hubard, Orbani\'c and Weiss showed that the automorphism group of the medial map $Me(\ma)$ of a map $\ma$ is isomorphic to the extended group $\mathcal{D}(\ma)$ of $\ma$, and used proper and improper self-dualities of the maps to characterize regular and 2-orbit medial maps, in terms of their symmetry type. In particular they showed that a medial map $Me(\ma)$ is regular if and only if $\ma$ is regular and self-dual. In \cite[Table 4]{MonGp_Self-Inv} we further observe that every 2-orbit symmetry type can be the medial map of a regular or a 2-orbit map.
In \cite{k-orbitM}, Orbani\'c, Pellicer and Weiss extended this to characterize the symmetry types of all medial maps of 2-orbit maps. They further proved that if $\ma$ is a $k$-orbit map, then $Me(\ma)$ is a $k$- or $2k$-orbit map, depending on whether or not $\ma$ is a self-dual map.

\section{Symmetry type graphs}
\label{sec:STG}

Let $\gr$ be the (edge-coloured) flag graph of a $k$-orbit map $\ma$, and $\mathrm{Orb}(\ma) := \{\oo_\Phi \mid \Phi \in \fl(\ma)\}$ the set of all the orbits of $\fl(\ma)$ under the action of $\Gamma(\ma)$.

We define the {\em symmetry type graph} $T(\ma)$ of $\ma$ to be the coloured factor pregraph of $\gr$ with respect to $\mathrm{Orb}(\ma)$. That is, the vertex set of $T(\ma)$ is the set of orbits $\mathrm{Orb}(\ma)$ of the flags of $\ma$ under the action of $\Gamma(\ma)$, and given two flag orbits $\oo_\Phi$ and $\oo_\Psi$, there is an edge
of colour $i$ between them if and only if there exists flags $\Phi'\in\oo_\Phi$ and $\Psi'\in\oo_\Psi$ such that $\Phi'$ and $\Psi'$ are $i$-adjacent in $\ma$. Edges between vertices in the same orbit shall factor into  semi-edges.

The simplest symmetry type graph arrises from regular maps. In fact, the symmetry type graph of a regular map has only one vertex and three semi-edges, one of each colour $0,1$ and $2$. Recall that the edges of $\ma$ are represented by 4-cycles of alternating colours 0-2 in $\gr$. Each of these 4-cycles should then factor into one of the five pregraphs in Figure~\ref{edges}.

\begin{figure}[htbp]
\begin{center}
\vspace{-1.5cm}
\includegraphics[width=12.5cm]{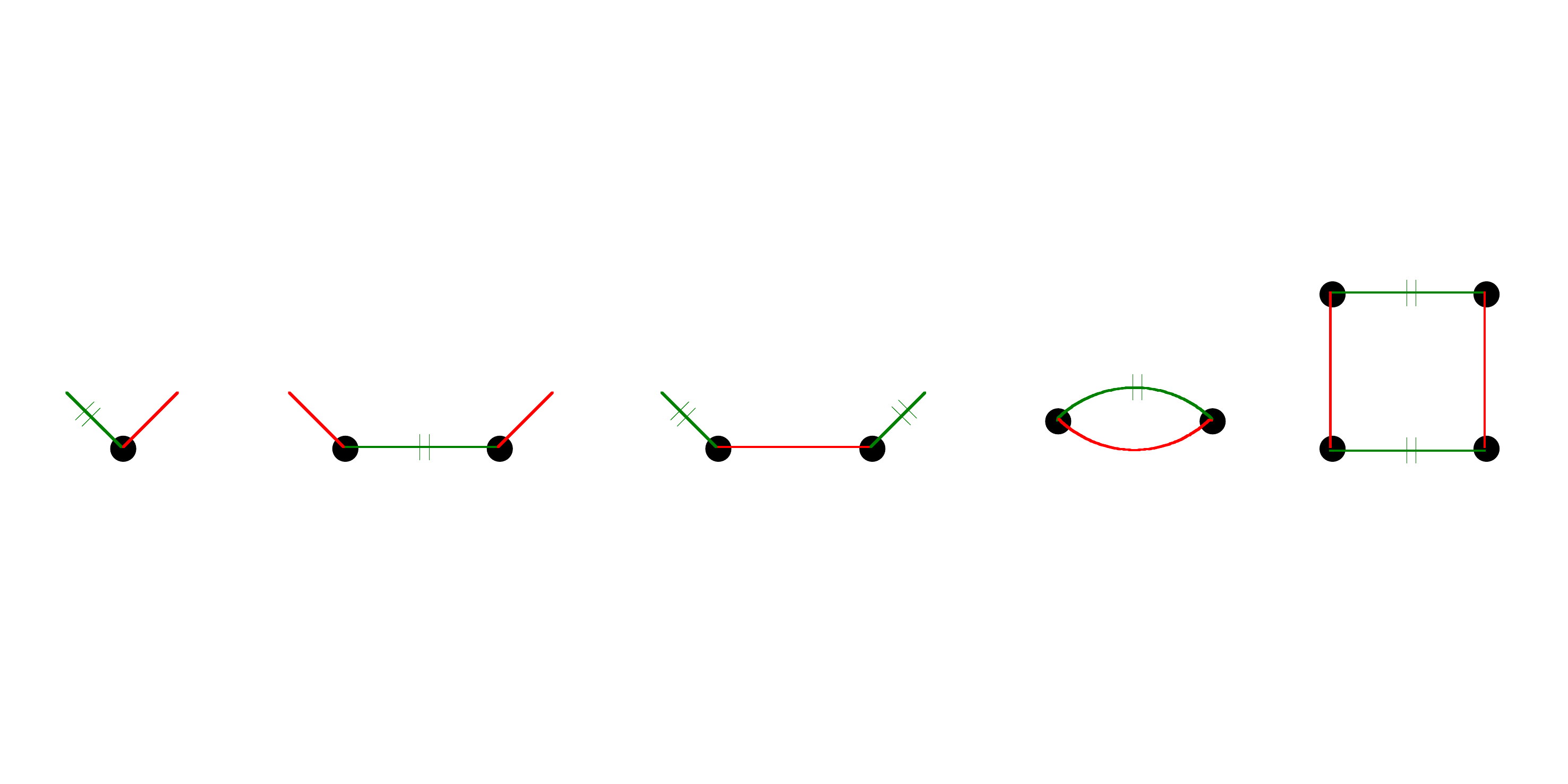}
\vspace{-2cm}
\caption{Possible quotients of 0-2 coloured 4-cycles.}
\label{edges}
\end{center}
\end{figure}

Clearly, if $\ma$ is a $k$-orbit map, then $T(\ma)$ has exactly $k$ vertices. Thus, the number of types of $k$-orbit maps depends on the number of 3-valent pregraphs on $k$ vertices that can be properly edge coloured with three colours and that the connected components of the 2-factor with colours $0$ and $2$ are always as in Figure~\ref{edges}.

It is then straightforward to see that there are exactly seven types of 2-orbit maps, shown in Figure~\ref{twoorbit}.
(The relations between some of these types, as shown in the figures, will be explained in Section \ref{duality}.)
These seven types of 2-orbit maps have been widely studied in different contexts, see for example \cite{rui} and \cite{2Polyh}; we follow \cite{2Polyh} for the notation of the types of the symmetry type maps with two flag orbits.

\begin{figure}[htbp]
\begin{center}
\vspace{-2cm}
\includegraphics[width=12.5cm]{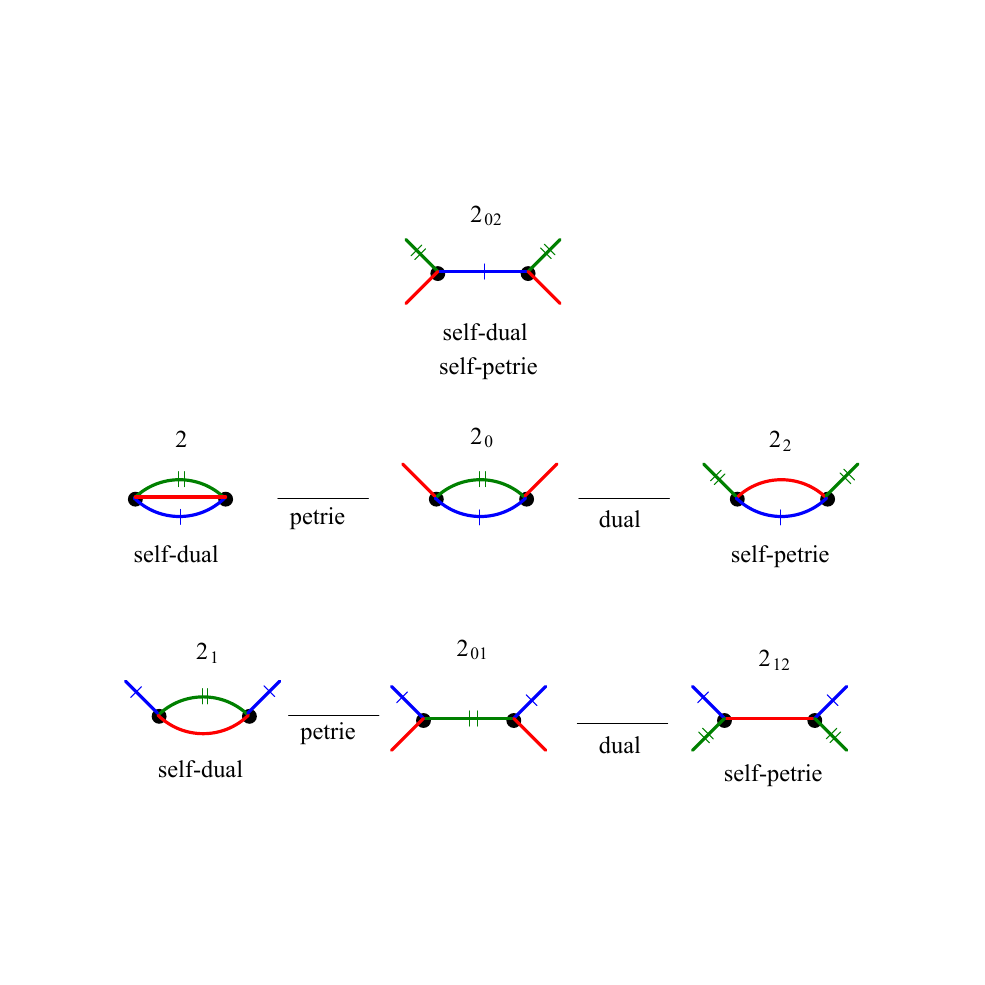}
\vspace{-2cm}
\caption{The seven symmetry type graphs of $2$-orbit maps.}
\label{twoorbit}
\end{center}
\end{figure}

In \cite{k-orbitM}, Orbani\'c, Pellicer, and Weiss studied all the types of $k$-orbit maps, for $k\leq 4$. It is likewise straightforward to see that there are only three 3-orbit maps types and these are shown in Figure~\ref{threeorbit}. For symmetry type graphs of maps of three and four orbits, we follow the notation of \cite{k-orbitM}.

\begin{figure}[htbp]
\begin{center}
\vspace{-8cm}
\includegraphics[width=12.5cm]{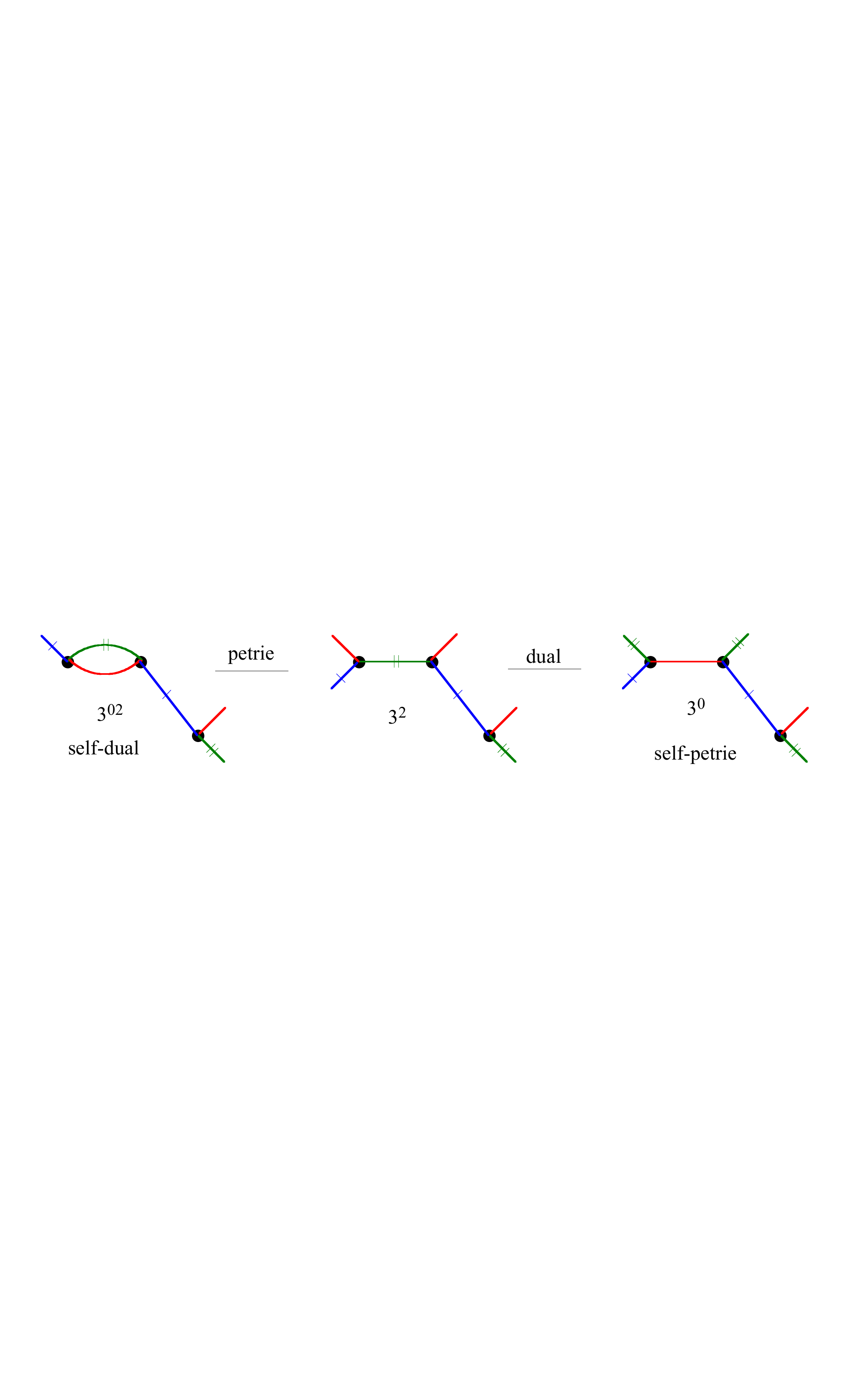}
\vspace{-9cm}
\caption{The three symmetry type graphs of $3$-orbit maps.}
\label{threeorbit}
\end{center}
\end{figure}

Because the automorphisms of a map $\ma$ preserves the colours of the edges of $\gr$, for every $\alpha \in \Gamma(\ma)$, $w \in \Mon(\ma)$ and $\Phi \in \fl(\ma)$ we have that $(\Phi^w) \alpha = (\Phi \alpha)^w$. We can therefore define the action of $\Mon(\ma)$ on the set $\mathrm{Orb}(\ma)$ as $\oo_\Phi \cdot w= \oo_{\Phi^w}$.
It is straightforward to see that this is a transitive action, as the one of $\Mon(\ma)$ on $\fl(\ma)$ is transitive.

The action of $\Mon(\ma)$ on $\mathrm{Orb}(\ma)$  can be easily seen on the symmetry type graph $T(\ma)$. In fact, in the same way as the words of $\Mon(\ma)$ can be seen as walks among the edges of $\gr$, they can be seen as walks among the edges of $T(\ma)$.
This immediately implies that, as $\Mon(\ma)$ acts transitively on $\mathrm{Orb}(\ma)$, the symmetry type graph $T(\ma)$ is connected.
A walk on $T(\ma)$ that starts at a vertex $\oo_\Phi$ and finishes at a vertex $\oo_\Psi$ corresponds to an element of $\Mon(\ma)$ that maps all the flags in the orbit $\oo_\Phi$ to flags in the orbit $\oo_\Psi$.
We can further see that closed walk among the edges of $T(\ma)$ that starts and finishes at a vertex $\oo_\Phi$ corresponds to an element of $\Mon(\ma)$ that permutes that flags of the orbit $\oo_\Phi$.

Given $i,j \in \{0,1,2\}$ with $i \neq j$, we say that an {\em $i$-$j$-walk} in $T(\ma)$ (or in $\gr$) is a walk along  edges of $T(\ma)$ (resp. $\gr$) of colours $i$ and $j$. The following lemma shall help us understand the orbits of the $k$-faces of a map $\ma$, in terms of the symmetry type graph of $\ma$.

\begin{lemma}
Let $\ma$ be a map with symmetry type graph $T(\ma)$.
For any two flags $\Phi$ and $\Psi$ of $\ma$, there is an $i$-$j$-walk in $T(\ma)$ between the vertices $\oo_\Phi$ and $\oo_\Psi$ of $T(\ma)$ if and only if $(\Phi)_k$ and $(\Psi)_k$ (with $k \neq i,j$) are in the same orbit of $k$-faces under the action of $\Gamma(\ma)$.
\end{lemma}

\begin{proof}
Suppose there is a $i$-$j$-walk in $T(\ma)$ between the vertices $\oo_\Phi$ and $\oo_\Psi$, and let $w \in \langle s_i, s_j\rangle$ be the associated element of $\Mon(\ma)$ corresponding to such walk. Then $\Phi^w \in \oo_\Psi$; that is, there exists $\alpha \in \Gamma(\ma)$ such that $\Phi^w\alpha=\Psi$. Now, by definition, $(\Phi\alpha)^w\in (\Phi\alpha)_k$.
On the other hand, as the action of $\Mon(\ma)$ commutes with the action of $\Gamma(\ma)$, then $(\Phi\alpha)^w=(\Phi^w)\alpha \in (\Phi^w\alpha)_k=(\Phi^w)_k\alpha$. That is, $\Psi \in (\Phi\alpha)_k \cap (\Phi^w\alpha)_k$. By Lemma~\ref{k-face} we then have that $(\Phi\alpha)_k=(\Phi^w\alpha)_k=(\Phi^w)_k\alpha$. Hence, $(\Phi\alpha)_k$ and $(\Phi^w)_k$ are in the same orbit of $k$-faces under $\Gamma(\ma)$. Moreover, $(\Phi\alpha)_k=(\Phi)_k\alpha$ implies that $(\Phi\alpha)_k$ and $(\Phi)_k$ belong to the same orbit under $\Gamma(\ma)$. We therefore can conclude that $(\Phi)_k$ and $(\Psi)_k=(\Phi^w\alpha)_k$ are in the same orbit of $k$-faces under the action of $\Gamma(\ma)$.

For the converse, let $\alpha \in \Gamma(\ma)$ be such that $(\Phi)_k\alpha = (\Psi)_k$.
That is, $\{\Phi^w \mid w \in \langle s_i, s_j \rangle\}\alpha = \{\Psi^u \mid u \in \langle s_i, s_j \rangle\}$.
Hence, there exists $w \in \langle s_i,s_j \rangle$ such that $\Phi^w\alpha=\Psi$. Now, $$\oo_\Phi \cdot w = \oo_{\Phi\alpha}\cdot w = \oo_{(\Phi\alpha)^w}=\oo_{\Phi^w\alpha}=\oo_\Psi.$$ Therefore $w\in\langle s_i, s_j \rangle$ induces a $i$-$j$-walk in $T(\ma)$ starting at $\oo_\Phi$ and finishing at $\oo_\Psi$.
\end{proof}

The following theorem is an immediate consequence of the above Lemma.

\begin{theorem}
\label{faceorbits}
Let $\ma$ be a map with  symmetry type graph
$T(\ma)$. Then, the number of connected components in the 2-factor of
$T(\ma)$ of colours $i$ and $j$, with $i,j \in \{0,1,2\}$ and $i \neq j$, determine
 the number of orbits of $k$-faces of $\ma$, where $k \in \{0,1,2\}$ is such that $k \neq i,j$.
 \end{theorem}

In particular, it follows from the above theorem, that an edge-transitive map $\ma$ has a symmetry type graph $T(\ma)$ with only one connected component of the 2-factor of colours $0$ and $2$. Since in the flag graph $\gr$ the edges of $\ma$ are given by 4-cycles, and quotient of a 4-cycle may have 1, 2, or 4 vertices (see Figure~\ref{edges}) it is immediate to see that an edge transitive map is  a 1-, 2-, or 4-orbit map (see \cite{GraverWatkins}).

In fact, there are 14  types of edge-transitive maps. These types were first studied by Graver and Watkins, in 1997 \cite{GraverWatkins}. After that, in 2001, Tucker, Watkins, and \v{S}ir\'a\v{n}, found that there exists a map for each type, \cite{EdgeT}.

The classification of symmetry type graphs of 3-orbit maps (see Figure~\ref{threeorbit}), together with Theorem~\ref{faceorbits} imply the following result.

\begin{corollary}
Every 3-orbit map has exactly two orbits of edges.
\end{corollary}

As pointed out in \cite{k-orbitM}, there are 22 types of 4-orbit maps. The 7 edge-transitive ones are shown in Figure~\ref{fourOrbitEdgeTrans}, while the 15 that are not edge-transitive are depicted in Figure \ref{fourOrbitNOEdgeTrans}.

\begin{figure}[htbp]
\begin{center}
\vspace{-2cm}
\includegraphics[width=12.5cm]{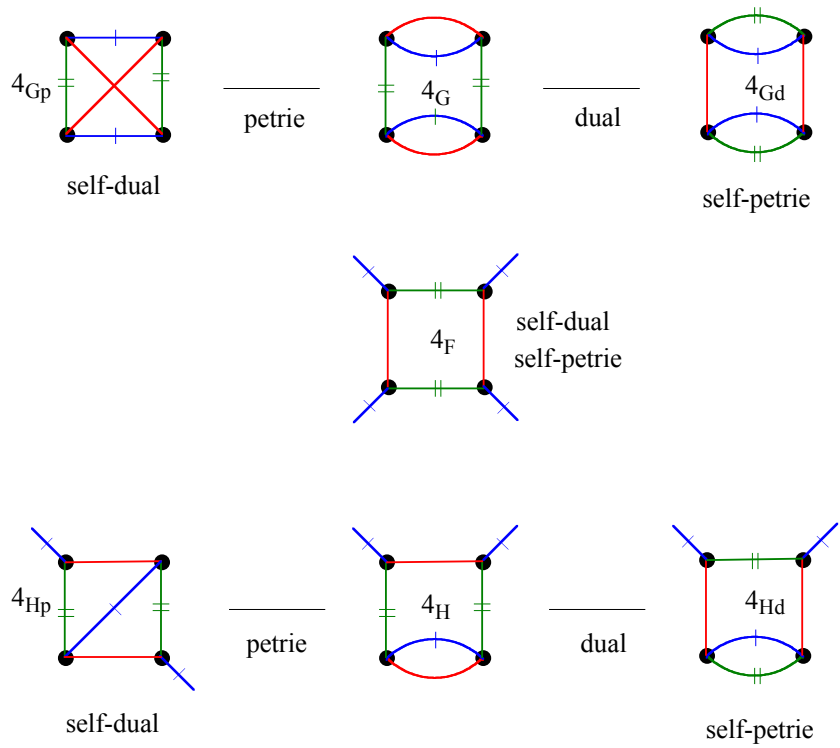}\vspace{-3cm}
\caption{The seven symmetry type graphs of edge-transitive $4$-orbit maps.}
\label{fourOrbitEdgeTrans}
\end{center}
\end{figure}

\begin{figure}[htbp]
\begin{center}
\includegraphics[width=12.5cm]{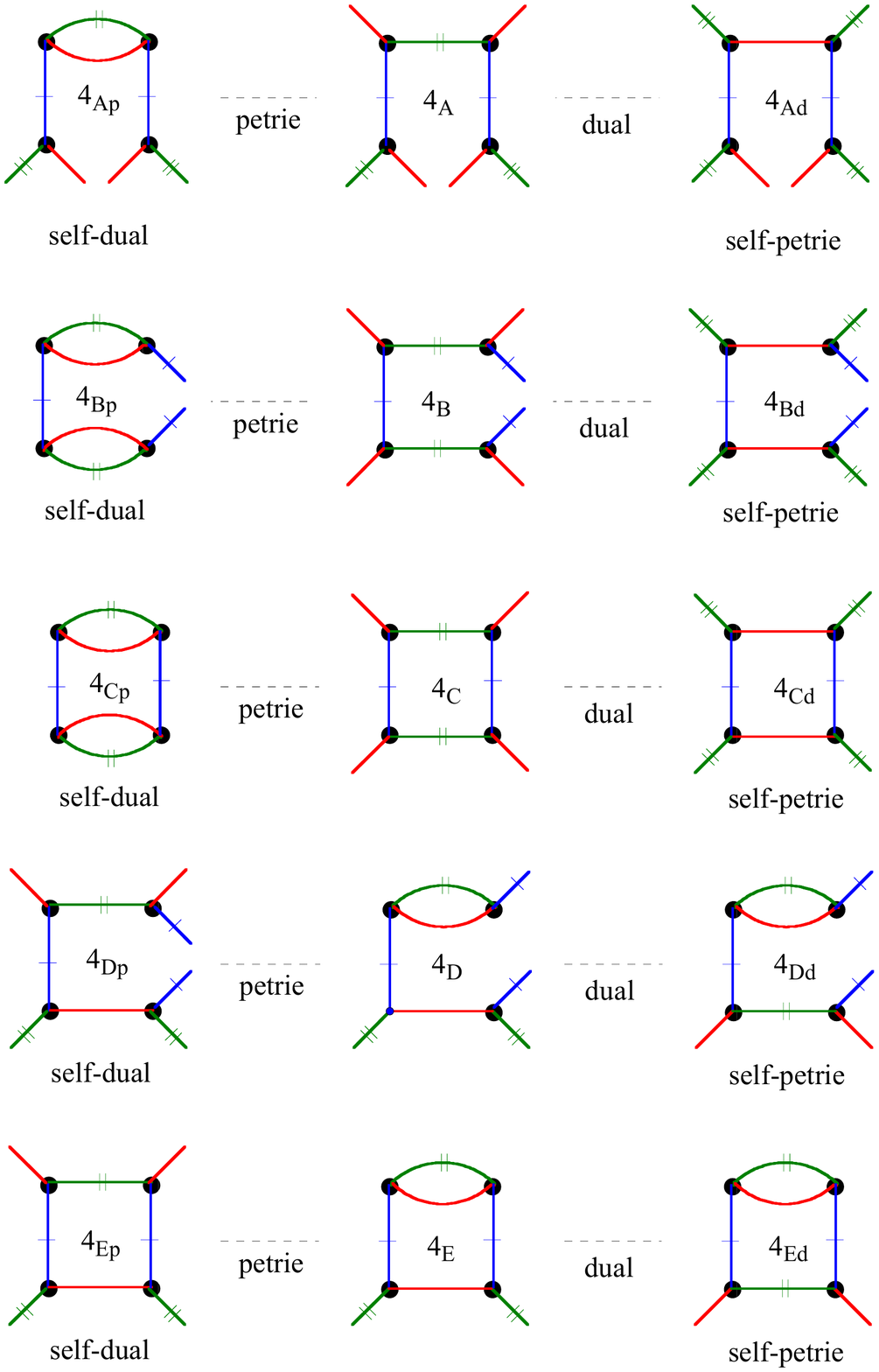}
\caption{The fifteen symmetry type graphs of $4$-orbit maps that are not edge-transitive.}
\label{fourOrbitNOEdgeTrans}
\end{center}
\end{figure}

Using the twenty two symmetry type graphs of 4-orbit maps, and the structure of the 2-factors of colours 0 and 2, one can see that there are thirteen different types of 5-orbit maps. Their symmetry type graphs are shown in Figure~\ref{fiveorbit}.

\begin{figure}[htbp]
\begin{center}
\vspace{-5cm}
\includegraphics[width=12.5cm]{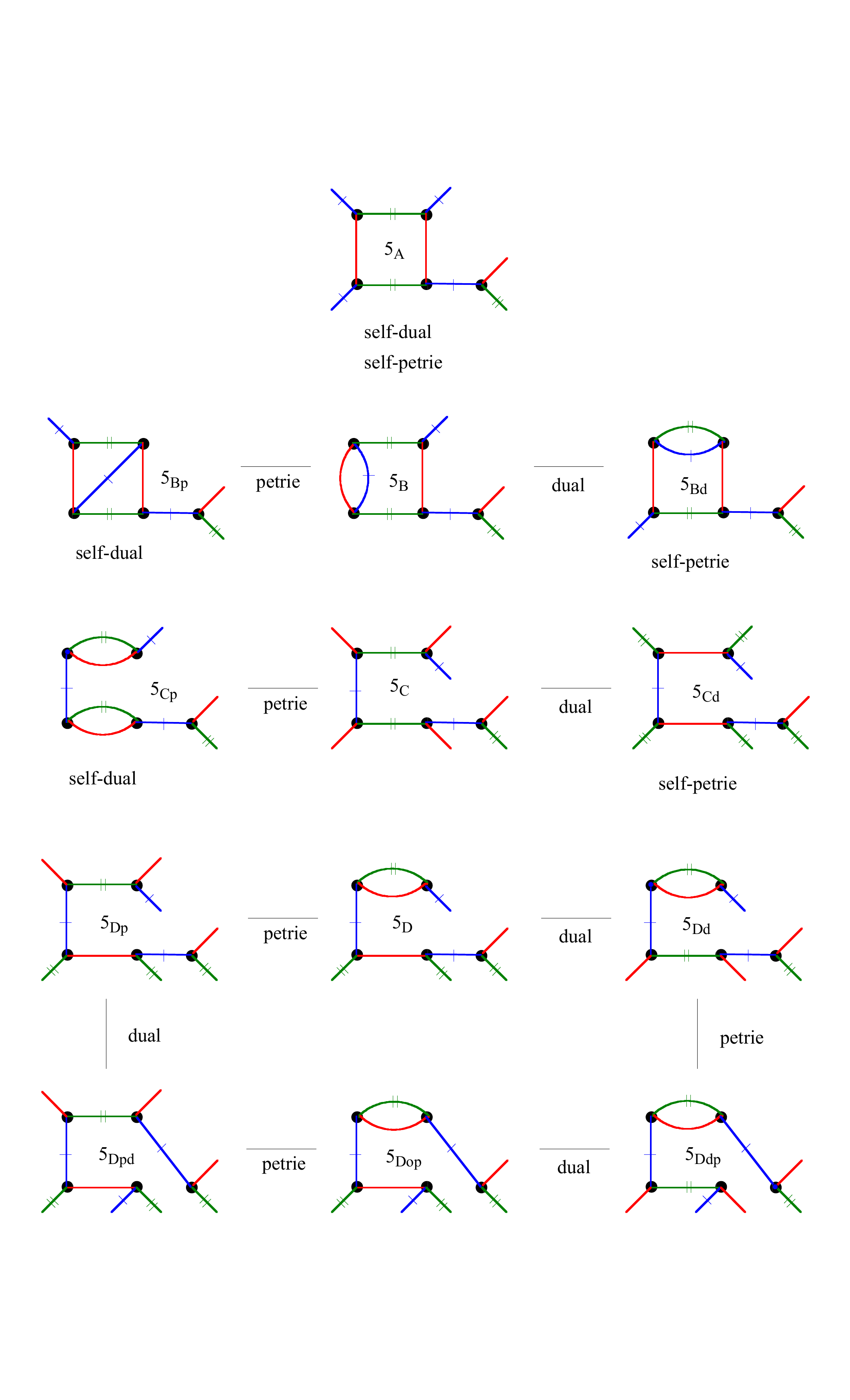}
\vspace{-2cm}
\caption{The thirteen symmetry type graphs of 5-orbit maps.}
\label{fiveorbit}
\end{center}
\end{figure}

\subsection{Dual and Petrie dual symmetry type graphs}
\label{duality}

Let $\delta$ be a duality of a map $\ma$ to its dual map $\ma^*$. Recall that $\delta$ can be regarded as a bijection between the vertices of $\gr$ and the vertices of $\mathcal{G_{\ma^*}}$ that sends edges of colour $i$ of $\gr$ to edges of colour $2-i$ of $\mathcal{G_{\ma^*}}$, for each $i \in \{0,1,2\}$. Then, the {\em dual type} of a symmetry type graph is simply a symmetry type graph with the same vertices and edges as the original, but with a permutation of the colours of its edges and semiedges in such a way that each colour $i \in \{0,1,2\}$ is replaced by the colour $2-i$. The following proposition is hence straightforward.

\begin{proposition}
If a map $\ma$ has symmetry type graph $T(\ma)$ then its dual $\ma^*$ has the dual of $T(\ma)$ as a symmetry type graph.
\end{proposition}

A symmetry type graph is said to be {\em self-dual} if it is isomorphic to its dual type.
Therefore, the symmetry type graph of a self-dual map is a {\em self-dual symmetry type graph}. However, the converse is not true. Not every map with a self-dual type is a self-dual map, for example, the cube and the octahedron are duals to each other (hence, they are not self-dual) and have the same symmetry type graph (as they are regular maps).

Note that by Lemma \ref{dualityaction}, each duality $\delta$ of a map $\ma$ induces a permutation $d$ of the vertices of $T(\ma)$, such that the edge colours 0 and 2 are interchanged.
In particular, the symmetry type graph of a properly self-dual map has a duality that fixes each of its vertices, while the symmetry type graph of an improperly self-dual map has a duality that moves at least two of its vertices.
Even more, since $\delta^2$ is an automorphism of $\ma$, then $\delta^2$ fixes each orbit of $\ma$.
Hence $d^2$ acts as the identity on the vertices of $T(\ma)$.
Therefore, for any duality $\delta$ of a self-dual map $\ma$ the corresponding duality $d$ of its symmetry type graph $T(\ma)$ is a polarity; i.e. a duality of order two.
However, in a similar way as before, the symmetry type graph does not posses all the information of the map.
That is, given a self-dual map $\ma$, its symmetry type graph $T(\ma)$ might not give us enough information on whether $\ma$ is properly or improperly self-dual.
An example of this is that chiral maps can be either properly or improperly self-dual (see \cite{isasia}), and hence the symmetry type graph of a chiral map accepts dualities that fix both vertices as well as dualities that interchange them.

Given a self-dual symmetry type graph $T(\ma)$ of a self-dual map $\ma$, the above paragraph incite us to add one edge (or semi-edge) of colour $D$ to each vertex of $T(\ma)$, representing the action of the dualities of $\ma$ on the flag orbits. The new pre-graph shall be called the {\em extended symmetry type graph} of the self-dual map $\ma$ and denoted by $\overline{T(\ma)}$. Since a self-dual regular map is always properly self-dual, then the extended symmetry type graph of a self-dual regular map consists of a vertex and four semi-edges, of colours $0,1,2$ and $D$, respectively.
Hence, as the distinguished generators $s_0, s_1, s_2$ of $\Mon(\ma)$ label the edges of $T(\ma)$, the edges of $\overline{T(\ma)}$ are labeled by $s_0, s_1, s_2$ and $d$.

Since for every flag $\Phi$ of a self-dual map $\ma$ and any duality $\delta$ of $\ma$ we have that $\Phi^1\delta=(\Phi\delta)^1$, the two factors of colours 1 and $D$ of $\overline{T(\ma)}$ are a factor of a 4-cycle. Furthermore, since $(\Phi\delta)^0\delta=(\Phi\delta^2)^{2}$, and $\delta^2 \in \Gamma(\ma)$, then  the path of $\overline{T(\ma)}$ coloured $D, 0, D$ starting at a given vertex $\oo_\Phi$, ends at $\oo_{\Phi^2}$;
that is, any path of colours $D,2,D,0$  finishes at the same vertex of $\overline{T(\ma)}$ that started.

We make here the remark that not every self-dual symmetry type accepts proper dualities, and that some symmetry types might accept more than one, essentially different, duality.
Every 2-orbit self-dual symmetry type admits both, a proper self-duality and an improperly self-duality. However, this is not always the case, for example, the only self-dual type of 3-orbit maps only admits a proper self-duality. Whenever an extended symmetry type graph has a properly self-duality, the colour $D$ of the graph consists of one semi-edge per vertex. In fact, we have the following proposition.

\begin{proposition}
Let $\ma$ be a self-dual map and let $T(\ma)$ be its symmetry type graph.
\begin{enumerate}
\item[a)]
If $T(\ma)$ has a connected component in its 2-factor of colours $0$ and $2$ that has exactly 4 vertices, then $\ma$ is improperly self-dual.
\item[b)]
If $T(\ma)$ has a connected component in its 2-factor of colours $0$ and $2$ that has exactly 2 vertices, one edge and 2 semi-edges, then $\ma$ is improperly self-dual.
\end{enumerate}
\end{proposition}

\begin{proof}
For $a)$, let $v_1, \dots, v_4$ the four vertices of a connected component in the 2-factor of $T(\ma)$ of colours 0 and 2. Without loss of generality let us assume that $\{v_1, v_2\}$ and $\{v_3, v_4\}$ are $0$-edges of $T(\ma)$, while $\{v_1, v_4\}$ and $\{v_2, v_3\}$ are $2$-edges of $T(\ma)$.  If $\ma$ is a properly self-dual map, then the colour $D$ of the extended graph $\overline{T(\ma)}$ consists of one semi-edge per vertex. Hence, the path $D,2,D,0$ takes the vertex $v_1$ to the vertex $v_3$, contradicting the fact that every $D,2,D,0$ starts and finishes at the same vertex. Therefore $\ma$ is improperly self-dual.
Part $b)$ follows in a similar way.
\end{proof}

The above proposition implies that if a map $\ma$ is properly self-dual, then the connected components in the 2-factor of $T(\ma)$ of colours 0 and 2 either have one vertex or have two vertices and a double edge between them.
Hence, up to five orbits, the types that admit properly self-dualities are types 1, 2, $2_1$, $2_{02}$, $3^{02}$, $4_{Ap}$, $4_{B_p}$, $4_{C_p}$ and $5_{C_p}$ (see Figures~\ref{twoorbit}, \ref{threeorbit}, \ref{fourOrbitNOEdgeTrans} and \ref{fiveorbit}).
Figure~\ref{PropTyp67} shows all self-dual symmetry type graphs with six and seven vertices that admit properly self-dualities.

\begin{figure}[htbp]
\begin{center}
\vspace{-2cm}
\includegraphics[width=12.5cm]{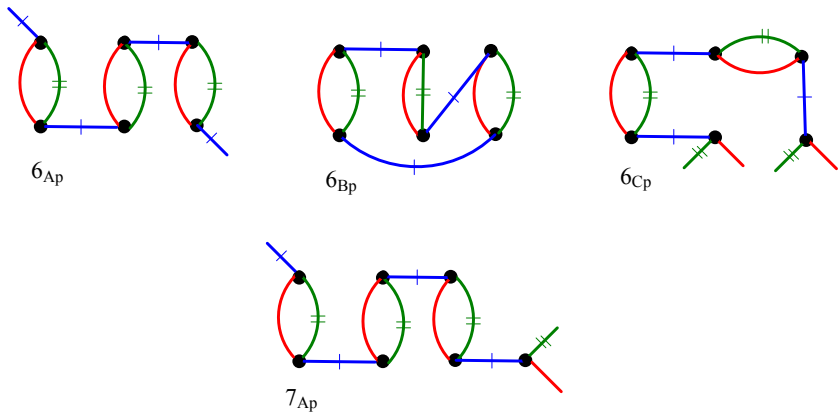}
\vspace{-2cm}
\caption{Symmetry type graphs with 6 and 7 orbits that admit proper self-dualitites.}
\label{PropTyp67}
\end{center}
\end{figure}

It should be now straightforward to see that the following corollary holds.

\begin{corollary}
If $k$ is even, there are exactly three extended symmetry type graphs with $k$ vertices admitting a proper self-duality. If $k$ is odd, there is exactly one extended symmetry type graph with $k$ vertices admitting a proper self-duality.
\end{corollary}

The number of extended symmetry type graphs having improper self-dualities is more convoluted. Figures~\ref{ExtImp234}, \ref{ExtImp56} and \ref{ExtImp7} shows the possible extended symmetry type graphs with at most seven orbits, having improperly self-dualitites.

\begin{figure}[htbp]
\begin{center}
\includegraphics[width=12.5cm]{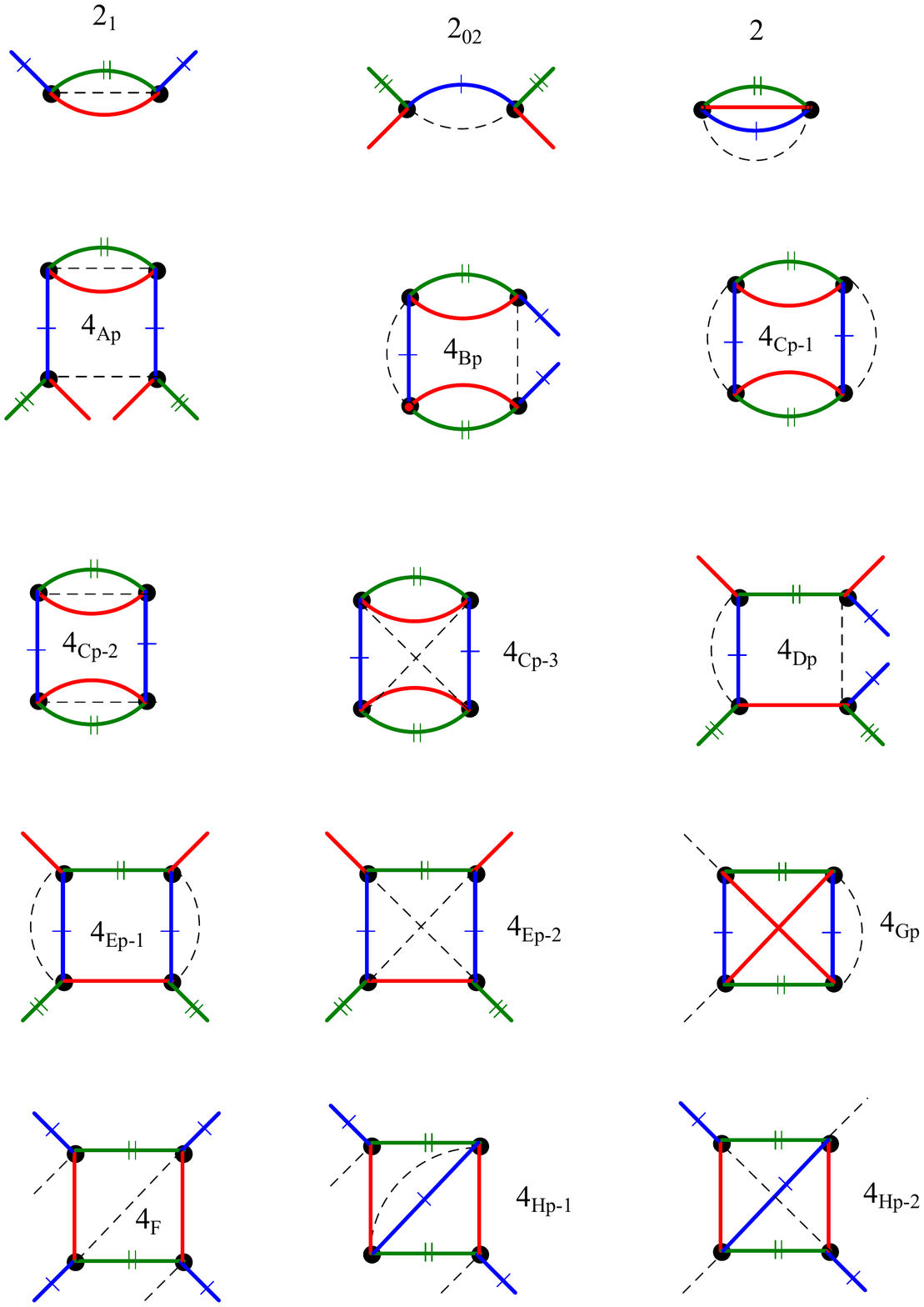}
\caption{Extended symmetry type graphs with at most 4 orbits, having improper self-dualitites.}
\label{ExtImp234}
\end{center}
\end{figure}

\begin{figure}[htbp]
\begin{center}
\includegraphics[width=12.5cm]{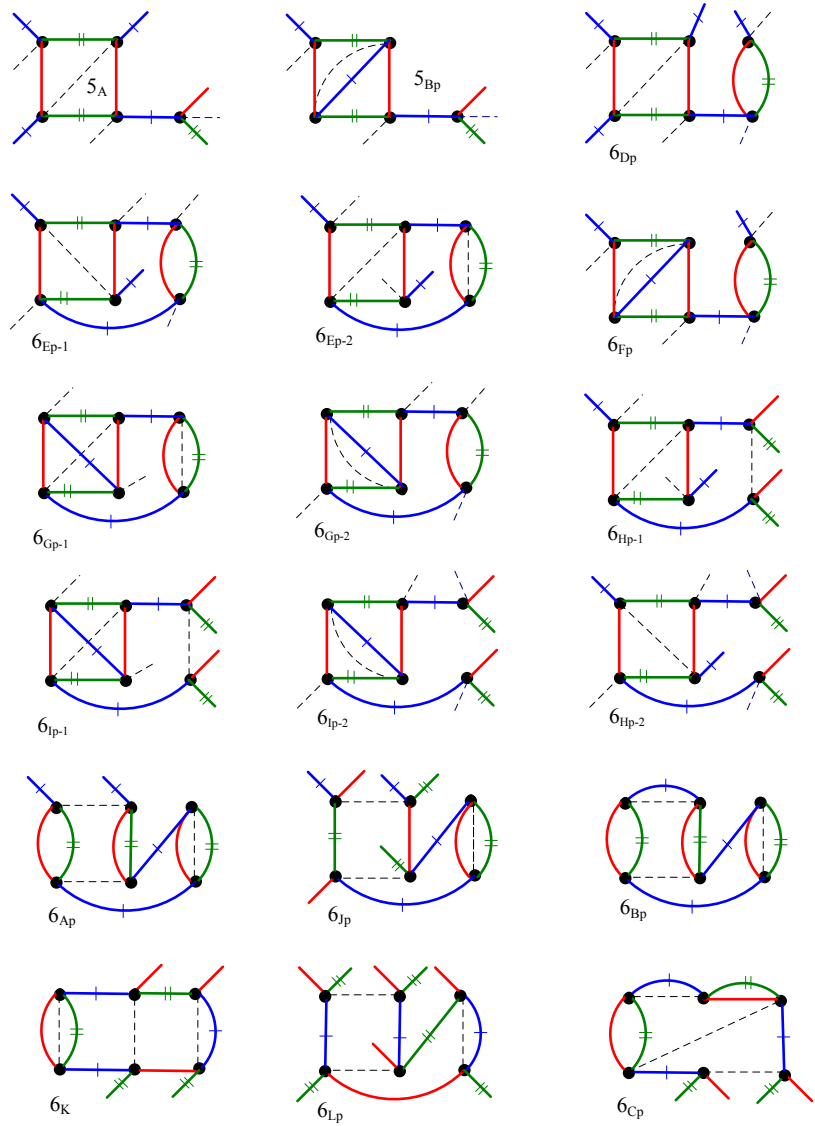}
\caption{Extended symmetry type graphs with 5 and 6 orbits, having improper self-dualitites.}
\label{ExtImp56}
\end{center}
\end{figure}

\begin{figure}[htbp]
\begin{center}
\includegraphics[width=12.5cm]{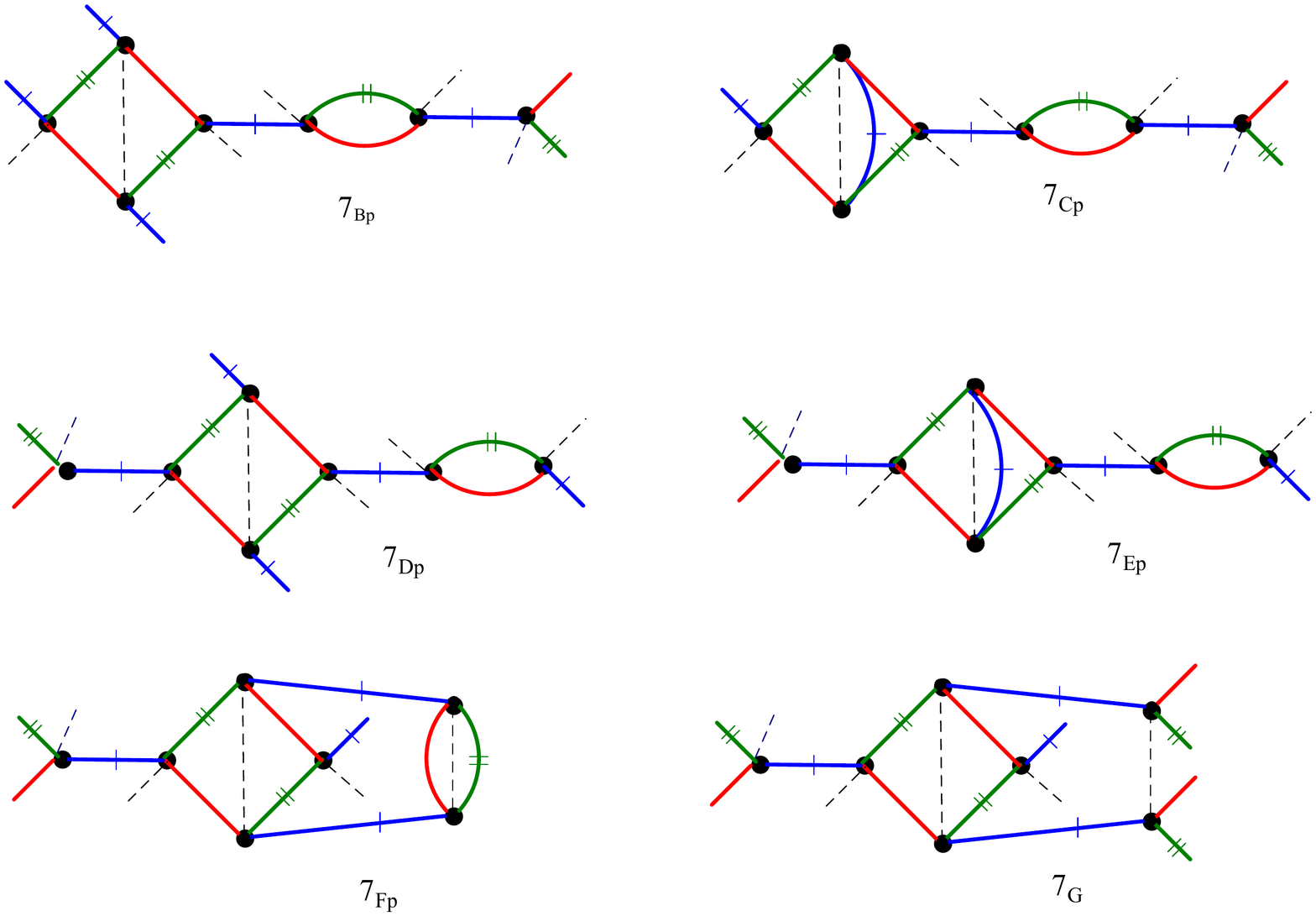}\vspace{-0.5cm}
\caption{Extended symmetry type graphs with 7 orbits, having improper self-dualitites.}
\label{ExtImp7}
\end{center}
\end{figure}

To complete this section, recall the bijection $\pi$ between the vertices of $\gr$ and the vertices of $\mathcal{G}_{\ma^P}$ induced by the Petrie dual of the map $\ma$,  preserves colours 1 and 2 and interchanges each (0,2)-path by colour 0 on the edges of the flag graph $\gr$. Then, similarly to the dual type, the {\em Petrie type} of a symmetry type graph $T(\ma)$ is a symmetry type graph with the same number of vertices of $T(\ma)$, which edges coloured 1 and 2 are preserved from $T(\ma)$, but any (0,2)-path in $T(\ma)$ is interchanged by an edge coloured by 0.

\begin{proposition}
If a map $\ma$ has symmetry type graph $T(\ma)$ then its Petrie-dual $\ma^P$ has the petrie-dual of $T(\ma)$ as symmetry type graph.
\end{proposition}

Similar to the self-dual symmetry type graph, a symmetry type graph it is said to be {\em self-Petrie} if it is isomorphic to its Petrie type.
The symmetry type graph of a self-Petrie map is a  self-Petrie symmetry type graph.

\newpage
\section{Medial symmetry type graphs and their enumeration}
\label{sec:med}

It is well-known that the medial of a tetrahedron is an octahedron and its medial is a cube-octahederon, see for instance \cite{CRP}.
While the former polyhedra are regular, the latter is only a 2-orbit edge-transitive as a map.
In the past several authors have observed that the medial of any regular map must be edge-transitive; in fact, in \cite{MonGp_Self-Inv} Hubard, Orbani\'c and Weiss showed that the medial of a regular map is either regular (if the original map is self-dual) or of type $2_{01}$.
The nature of edge-transitive tessellations have been studied by Graver and Watkins, \cite{GraverWatkins}. They were the first to determine all 14 different symmetry types of edge-transitive maps. Later, \v Siran, Tucker and Watkins \cite{EdgeT} have provided examples of maps from each of the 14 types.
Several authors have been trying to determine the nature of edge-transitive maps that are medial maps, i.e. maps that are medials of other maps.
For instance, Lemma 2.2 in \cite{EdgeT} lists six symmetry types that can be edge-transitive medials of
edge-transitive maps. In \cite{MonGp_Self-Inv} it is shown that there are, in fact, seven such symmetry types.
In Table 2 of \cite{k-orbitM} the authors give 10 symmetry types of edge-transitive maps that may be medials of other, not necessarily edge-transitive maps. Unfortunately, they miss the fact that four other edge-transitive types may also be medials.

In this section, we make use of symmetry type graphs and extended symmetry type graph, and an operation on them to obtain medial symmetry type graphs. We further enumerate all the medial types with at most 7 vertices and show that indeed all edge-transitive types are medials.

We shall say that a {\em medial symmetry type graph} is the symmetry type graph of a medial map, denoted by $T(Me(\ma))$. In what follows we classify all medial symmetry type graphs with at most 7 vertices, that is, the possible symmetry type graphs of medial $k$-orbit maps, with $k\leq 7$. To this end, we develop basic operations on the symmetry type graphs as well as on the extended symmetry type graphs, based on the flag graphs of a map and its medial.

If a map $\ma$ is not a self-dual map, we may think of the vertices of the medial symmetry type graph $T(Me(\ma))$,  as those obtained by two copies of the vertices of the symmetry type graph $T(\ma)$. 
As with the flag graph, given a vertex $\oo_\Phi$ of $T(\ma)$ we can write its corresponding two copies  in $T(Me(\ma))$ as $(\oo_\Phi, 0)$ and $(\oo_\Phi, 2)$. Note that the edges between these copies of the vertices of $T(\ma)$ must respect the colour adjacency of the flags in the flag graph of $Me(\ma)$.
Then, we can follow the same ``standard" algorithm shown in the Figure \ref{med-flag} to determine the adjacencies between the vertices of $T(Me(\ma))$.
In other words, the vertices $(\oo_\Phi,0)$ and $(\oo_\Phi,2)$ are adjacent by an edge of colour 2;
for $i=0,2$ there is an edge of colour 0 between $(\oo_\Phi,i)$ and $(\oo_\Psi,i)$ if and only if $\oo_\Phi$ and $\oo_\Psi$ are adjacent by the colour 1.
Finally, there is an edge of colour 1 between $(\oo_\Phi,i)$ and $(\oo_\Psi,i)$ if and only if $\oo_\Phi$ and $\oo_\Psi$ are adjacent by the colour 0 or 2.

Hence, if a $k$-orbit map $\ma$ is not a self-dual map, the medial symmetry type graph $T(Me(\ma))$ of $Me(\ma)$  (obtained as it was described in the paragraph above) has $2k$ vertices.
On the other hand, when $\ma$ is a self-dual $k$-orbit map, to obtain its medial symmetry type graph with $k$ vertices,
we first proceed as in the above paragraph and then take into consideration the extended symmetry type graph $\overline{T(\ma)}$.
In this case we shall identify each vertex of the form $(\oo_\Phi,0)$ with a vertex of the form $(\oo_\Psi,2)$ whenever  $\oo_\Phi$ and $\oo_\Psi$ are adjacent by the colour $D$ in $\overline{T(\ma)}$.
Thus, the edges of colour 2 in $T(Me(\ma))$ are determined by the respective duality $\delta$ on the self-dual map $\ma$.
Consequently the colours of the edges of $T(Me(\ma))$ can be defined by the following involutions:

\begin{eqnarray*}
S_0 =& s_1,& \\
S_1 =& s_0 &(\mbox{or } s_2), \\
S_2 =& d.&
\end{eqnarray*}

Note that if $Me(\ma)$ is a $k$-orbit map, with $k$ odd, then $\ma$ is a self-dual $k$-orbit map. However, if $k$ is even, then $\ma$ is either a $k$- or a $k/2$-orbit map. Hence, to obtain all medial symmetry type graphs with at most 7 vertices, one has to apply the above operations to all symmetry type graphs with at least 3 vertices, as well as to all extended symmetry type graphs with at most 7 vertices.

In \cite[Table 2]{k-orbitM} are given the symmetry types of medials coming from 1- and 2-orbit maps. Following the algorithm described above, in the left table of Table~\ref{med1-3,4-5} we repeat the information contained in \cite[Table 2]{k-orbitM} and give the symmetry type of medials coming from 3-orbit maps. In the right table of Table~\ref{med1-3,4-5} and in both tables of Table~\ref{med6-7} are given the symmetry type of medials coming from k-orbit self-dual maps, for $4 \leq k \leq 7$. In the second row, of all tables in Tables \ref{med1-3,4-5} and \ref{med6-7}, ``P'' stands for properly self-dual, ``I" for improperly self-dual and ``N" for no duality needed, the number after the I, in the cases it exists, stands for the type of improperly duality that the map possesses. All medial types with at most 5 vertices are already given in Figures~\ref{twoorbit}-\ref{fiveorbit}; medial types with 6 and 7 vertices are given in Figures~\ref{6d6m} and \ref{medial7orbit}, respectively.

\begin{table}[h!]
\centering
\begin{tabular}{|c|c|c|c||c|c|c|c|c|c|c|}
\hline
Sym type   & \multicolumn{3}{|c|}{Sym type}               &  Sym type   & \multicolumn{4}{|c|}{Sym type}                \\
 of $\ma$   &  \multicolumn{3}{|c|}{of $Me(\ma)$ }      &  of $\ma$   &  \multicolumn{4}{|c|}{of $Me(\ma)$ }        \\

\hline
   Duality    &         P        &         I       &        N               &    Duality    &        P        &      I-1       &        I-2       &   I-3  \\  

\hline
        1        &       1          &       ---       &  $2_{01}$        & $4_{A_p}$ & $4_{B_d}$ & $4_{H_d}$ &       ---         &    ---   \\
        2        &    $2_2$     &         2       &  $4_G$             & $4_{B_p}$ & $4_{A_d}$ & $4_{E_d}$ &       ---         &     ---   \\
    $2_0$    &       ---        &       ---       &  $4_H$           & $4_{C_p}$ & $4_{C_d}$ & $4_{C_p}$ & $4_{G_d}$ & $4_{G_p}$\\
    $2_1$    &  $2_{02}$  &   $2_0$    &  $4_C$              & $4_{D_p}$ &      ---       & $4_{D_d}$ &      ---         &     ---    \\
    $2_2$    &       ---        &       ---       &   $4_H$            & $4_{E_p}$  &      ---       & $4_{B_p}$ & $4_{H_p}$ &     ---   \\
 $2_{01}$  &       ---        &       ---      &  $4_A$              &    $4_F$      &      ---       &     $4_A$    &      ---          &     ---   \\
 $2_{02}$  &  $2_{12}$  &    $2_1$   &  $4_F$              & $4_{G_p}$  &      ---      &     $4_E$     &      ---          &     ---   \\
 $2_{12}$  &       ---        &       ---      & $4_A$               & $4_{H_p}$  &      ---      & $4_{A_p}$ & $4_{E_p}$  &     ---   \\
   $3^0$     &       ---        &       ---      &  $6_D$              &    $5_A$      &      ---      & $5_{D_{pd}}$ &     ---      &     ---   \\
   $3^2$     &       ---        &       ---      &  $6_D$              & $5_{B_p}$  &       ---     & $5_{D_{op}}$ &    ---       &     ---   \\
 $3^{02}$  &   $3^0$     &       ---      &  $6_M$              & $5_{C_p}$ & $5_{C_d}$ &         ---         &    ---      &     ---   \\

 \hline
 \end{tabular}
\caption{Medial symmetry types from 1-, 2-, 3-orbit maps (in the left), and from 4-, and 5-orbit maps (in the right).}
\label{med1-3,4-5}
\end{table}

\begin{table}[h!]
\centering
\begin{tabular}{|c|c|c|c||c|c|c|c|c|c|c|}
\hline
Sym type   & \multicolumn{3}{|c|}{Sym type}               &  Sym type   & \multicolumn{2}{|c|}{Sym type}                \\
 of $\ma$   &  \multicolumn{3}{|c|}{of $Me(\ma)$ }      &  of $\ma$   &  \multicolumn{2}{|c|}{of $Me(\ma)$ }        \\
\hline
    Duality     &         P        &         I-1       &         I-2          &    Duality    &         P         &           I             \\  
\hline
 $6_{A_p}$ & $6_{C_d}$ & $6_{N_p}$ &         ---          &    $6_K$      &         ---         &  $6_{Q}$       \\ 
 $6_{B_p}$ & $6_{B_d}$ &     $6_N$     &         ---          &  $6_{L_p}$ &        ---         & $6_{D_p}$     \\
 $6_{C_p}$ & $6_{A_d}$ & $6_{M_p}$ &         ---          &  $7_{A_p}$ & $7_{A_d}$ &          ---          \\
 $6_{D_p}$ &       ---         & $6_{L_p}$  &         ---          &  $7_{B_p}$ &          ---      &     $7_H$        \\
 $6_{E_p}$ &       ---         & $6_{O_p}$ &   $6_{I_d}$   &  $7_{C_p}$ &          ---      & $7_{H_p}$     \\
 $6_{F_p}$ &       ---         &     $6_L$     &         ---          &  $7_{D_p}$ &          ---      &      $7_I$        \\
 $6_{G_p}$ &       ---         & $6_{G_d}$ &     $6_O$       &  $7_{E_p}$ &          ---      &  $7_{I_p}$     \\
 $6_{H_p}$ &       ---         & $6_{H_d}$ &     $6_P$        &  $7_{F_p}$ &          ---      &  $7_{E_d}$    \\
 $6_{I_p}$ &       ---         & $6_{E_d}$ &   $6_{P_p}$    &    $7_{G}$   &          ---      & $7_{J}$         \\
 $6_{J_p}$ &       ---         & $6_{Q_p}$ &         ---          &                     &                    &                       \\ 
\hline
 \end{tabular}
\caption{Medial symmetry types from 6- and 7-orbit maps.}
\label{med6-7}
\end{table}

\begin{figure}[htbp]
\begin{center}
\includegraphics[width=12.5cm]{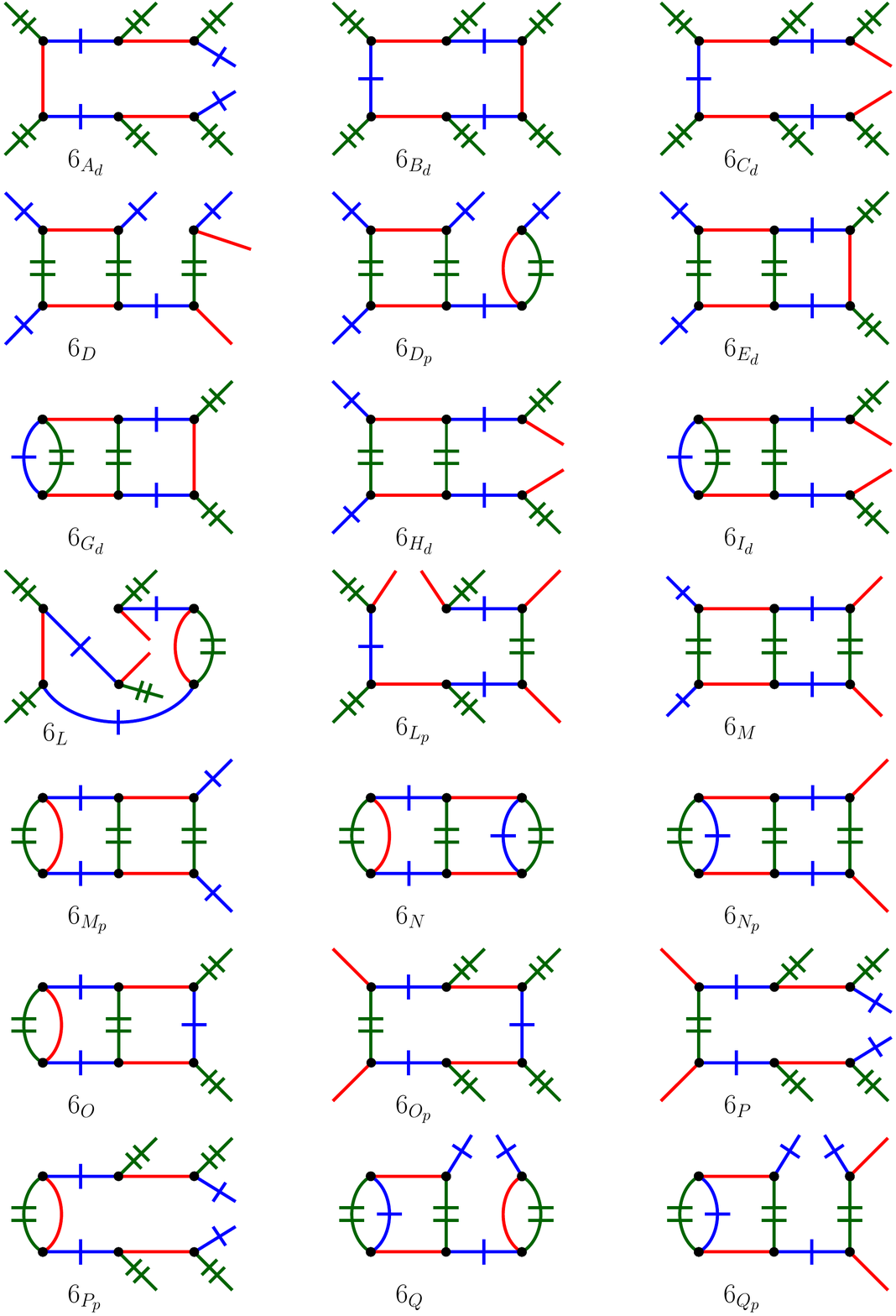}
\caption{Medial symmetry types with 6 vertices.}
\label{6d6m}
\end{center}
\end{figure}

\begin{figure}[htbp]
\begin{center}
\vspace{-3.5cm}
\includegraphics[width=12.5cm]{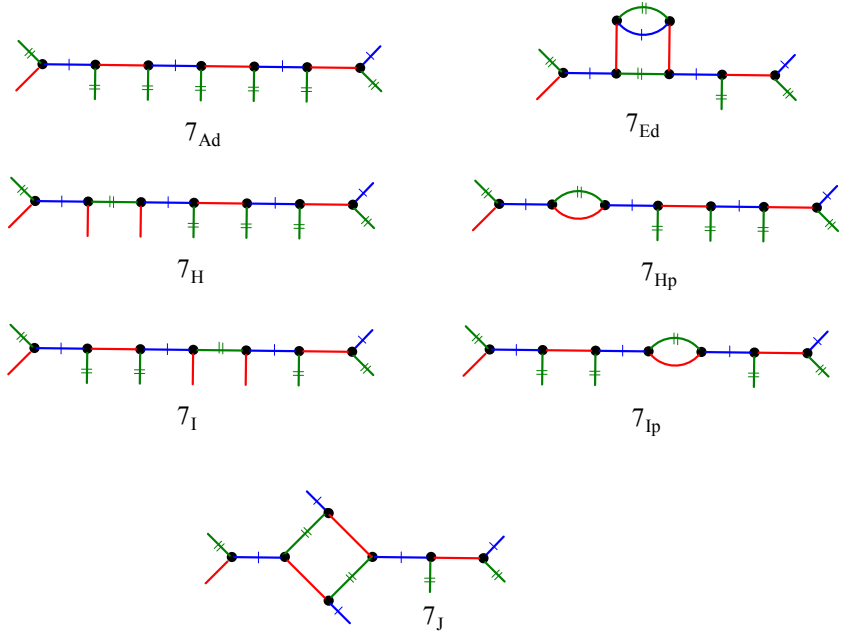}
\vspace{-3.5cm}
\caption{Medial symmetry types with 7 vertices.}
\label{medial7orbit}
\end{center}
\end{figure}

\newpage
To show that each of the 14 edge-transitive symmetry type graphs is the symmetry type graph of a medial map we shall, for each type, give an example. That is, for each of the 14 edge-transitive types we give a map whose medial map is of the given type. 

Orbani\'c, in \cite{ETdatabase} generated a data base of small non-degenerated edge transitive maps. His data base contains small non-degenerated edge transitive maps of types 1, $2_{12}$, $2_2$, $4_F$, $4_{Hd}$ and $4_{Gd}$; the remaining types can be obtained from these one by making use of the Petrie and dual operations (see Figures~\ref{twoorbit} and~\ref{fourOrbitEdgeTrans}).

Using the tables in Table~\ref{med1-3,4-5}, we obtain Table \ref{edge-t_med}. The first column of Table \ref{edge-t_med} lists all the candidate types for maps that could, by applying medial operation, yield the maps with edge-transitive types. The second column indicates which type of duality the map in the first column should have to obtain the medial type in the third column. In the second column the number after the Improper, in the cases it exists, stands for the type of improperly duality that the map possesses, see right table of Table~\ref{med1-3,4-5}.

\begin{table}[h!]
\begin{center}
\begin{tabular}{|c|c|c|c||c|c|c|c|c|c|c|}
\hline
           $T(\ma)$                    &    Duality    &    $T(Me(\ma))$   \\

\hline
\multirow{2}{*}{1}               &    Proper     &           1               \\
                                              &     None      &     $2_{01}$        \\
\hline
\multirow{3}{*}{$2_{02}$}  &    Proper     &     $2_{12}$        \\
                                              &  Improper   &        $2_1$          \\
                                              &     None      &        $4_F$          \\
\hline
                           $2_1$         &  Improper  &         $2_0$          \\
\hline
\multirow{3}{*}{$2$}            &    Proper     &        $2_2$          \\
                                               &  Improper   &          $2$          \\
                                               &     None      &        $4_G$          \\
\hline
\multirow{2}{*}{$4_{C_p}$} & Improper-2 &   $4_{G_d}$      \\
                                               & Improper-3 &   $4_{G_p}$      \\
\hline
                     $2_{0},2_{2}$  &      None      &     $4_H$           \\
\hline
                       $4_{A_p}$      &  Improper    &   $4_{H_d}$      \\
\hline
                       $4_{E_p}$      & Improper-2 &   $4_{H_p}$      \\
\hline

 \end{tabular}
\caption{Edge-transitive medial symmetry types}
\label{edge-t_med}
\end{center}
\end{table}

\begin{theorem}\label{EdgeT-Med}
Each of the 14 edge-transitive symmetry type graphs is the symmetry type graph of a medial map.
\end{theorem}

\begin{proof}
Following \cite{k-orbitM}, a map $\ma$ can be regarded as the action of $\Mon(\ma)$ on the set of flags $\fl(\ma)$. 
This in turn is equivalent to consider the action of $\mathcal{C}:=\langle r_0, r_1, r_2 \mid r_0^2=r_1^2=r_2^2=(r_0r_2)^2=id\rangle$ on the cosets $\mathcal{C} / N$ (where $N \leq \mathcal{C}$ with index $[\mathcal{C} : N]$ finite, is thought as the stabilizer of a given flag $\Phi$ of $\ma$), given by a homomorphism $\chi: \mathcal{C} \to Sym(\mathcal{C}/N)$. 
In this notation, the monodromy group of $\ma$ is the image of $\chi$, together with the generators $\chi(s_i)$, $i=0,1,2$.

According to \cite[Sec. 4.1]{k-orbitM}, the medial map $Me(\ma)$ of $\ma$ can be regarded as the action of  $\mathcal{C}$  on the cosets $\mathcal{C} / f^{-1}\psi(N)$, where $\psi$ is the isomorphism mapping $(r_0, r_1, r_2) \mapsto (s_1, s_0, s_{212})$ and $f:\mathcal{C}  \to \mathcal{C}_4 = \langle s_0, s_1, s_2\vert s_0^2, s_1^2, s_2^2, (s_0s_2)^2, (s_1s_2)^4\rangle$ is the natural epimorphism.
In other words, a map $\ma$ given by three involutions $s_0$, $s_1$ and $s_2$ generating the monodromy group $\mathrm{Mon}(\ma) = \langle s_0, s_1, s_2\rangle$ is a medial of a map if and only if the $\Mon(\ma)$ is a quotient of  $\mathcal{C}_4$, 
the index of the subgroup $H := \langle s_1, s_0, s_2s_1s_2\rangle \leq \mathrm{Mon}(M)$ is 2, and any stabilizer $N \leq \mathrm{Mon}(\ma)$ of a flag is contained in $H$.

Then the de-medialized map is defined by taking the action of $H$ on the cosets of $N \leq H$ and relabeling generators of $H$ in the respective order by $s_0$, $s_1$ and $s_2$. Note that the dual of the result (yielding the same medial map) could be obtained by taking the conjugate stabilizer $s_2Ns_2$ instead of $N$.

Using this method, software package {\sc Magma} and the database of small non-degenerate edge-transitive maps \cite{ETdatabase}, all the examples  supporting options in Table \ref{edge-t_med} are calculated and summarized in Table \ref{edge-t_med_ex}. It is not claimed that they are minimal examples, though we tried to choose minimal such cases where both an original map and its medial are non-degenerate and both with edge multiplicity 1.
\end{proof}

\begin{sidewaystable}[htp]
\begin{tabular}{|c|c|c|c|c|c|c|c|c|c|c|c|}
\hline
\multicolumn{6}{|c|}{\bf Edge-transitive $\mathrm{Med}(\ma)$} & \multicolumn{6}{|c|}{\bf Example map $\ma$} \\
\hline
{\bf InitType}&{\bf  ID} & {\bf Tran}   & {\bf  MType} & {\bf Genus} & {\bf Symbol} & {\bf Type} & {\bf $|V|$} & {\bf $|E|$} & {\bf $|F|$} & {\bf $|P|$} & {\bf Map symbol} \\
\hline
1 & 41 & P & 1 &  1 & $\langle 4; 4; 8\rangle$ & 1 & 8 & 16 & 8 & 8 & $\langle  4 ;  4 ;  4 \rangle$\\
\hline
$2$ &113 & D &  $2_{01}$ & 0 & $\langle 4; 3,4; 8\rangle$ & 1 & 6 & 12 & 8 & 4 & $\langle  4 ;  3 ;  6 \rangle$\\
\hline
$2$ & 8335 &  & $2_{12}$ & 25 & $\langle 4,4; 10; 10\rangle$ & $2_{02}$ & 16 & 80 & 16 & 40 & $\langle  10 ;  10 ;  4 \rangle$\\
\hline
$2$ &504 & DP &  $2_1$ & 1 & $\langle 4; 4; 6,8\rangle$ & $2_{02}$ & 12 & 24 & 12 & 2 & $\langle  4 ;  4 ;  24 \rangle$\\
\hline
$3$ & 180 & P & $4_F$ & -8 & $\langle 4,4; 12,4; 6,12\rangle$ & $2_{02}$ & 9 & 18 & 3 & 3 & $\langle  4 ;  12 ;  12 \rangle$\\
\hline
$2ex$ &456 & D & $2_0$ & 37 & $\langle 4; 8; 8\rangle$ & $2_{1}$ & 36 & 144 & 36 & 24 & $\langle  8 ;  8 ;  12, 12 \rangle$\\
\hline
$2ex$ & 21 &  & $2_2$ & 7 & $\langle 4; 7; 14\rangle$ & 2 & 8 & 28 & 8 & 14 & $\langle  7 ;  7 ;  4 \rangle$\\
\hline
$2ex$ &2 & DP &  2& 1 & $\langle 4; 4; 10\rangle$ & 2 & 5 & 10 & 5 & 2 & $\langle  4 ;  4 ;  10 \rangle$\\
\hline
$5$ & 13 & D  & $4_G$ & 1 & $\langle 4; 6,3; 12\rangle$ & 2 & 14 & 21 & 7 & 3 & $\langle  3 ;  6 ;  14 \rangle$\\
\hline
$5$ & 301 &  & $4_{Gd}$ & 19 & $\langle 4,4; 8; 8\rangle$ & $4_{Cp}$ & 18 & 72 & 18 & 12 & $\langle  8 ;  8 ;  12, 12 \rangle$\\
\hline
$5$ & 275 & DP & $4_{Gp}$ & 18 & $\langle 4; 8; 8,8\rangle$ & $4_{Cp}$ & 17 & 68 & 17 & 4 & $\langle  8 ;  8 ;  34, 34 \rangle$\\
\hline
$4$ & 450 & D & $4_{H}$ & -8 & $\langle 4; 7,4; 8\rangle$ & $2_{2}$ & 14 & 28 & 8 & 8 & $\langle  4 ;  7 ;  7 \rangle$\\
\hline
$4$& 17200 &  & $4_{Hd}$& 26 & $\langle 4,4; 8; 8\rangle$ & $4_{Ap}$ & 25 & 100 & 25 & 20 & $\langle  8 ;  8 ;  10, 10 \rangle$\\
\hline
$4$&7496 & DP & $4_{Hp}$ & 19 & $\langle 4; 8; 8,8\rangle$ & $4_{Ep}$ & 18 & 72 & 18 & 36 & $\langle  8 ;  8 ;  4 \rangle$\\
\hline
\end{tabular}
\caption{Examples of edge-transitive maps of all 14 types that are medials. 
\label{edge-t_med_ex}}
\vspace{1cm}
The table is divided into two halves and each line represents one example. The first half contains data needed to retreive a map $Me(\ma)$ from the database \cite{ETdatabase}. Three parameters are needed: a type ({\bf MType}, one of the types 1, 2, 2ex, 3, 4 or 5 according to  \cite{GraverWatkins}), an identifier within a subdatabase for the type  (column {\bf ID}) and a sequence of operations  (column {\bf Tran}, where D stands for dual and P for Petrie-dual; note: operations compose like functions). The retreived map (denoted by $Me(\ma)$)  is the medial of the map $\ma$ whose type, number of vertices, edges, faces, Petrie-polygons and symbol are in the seventh to twelfth columns of the table (second half). 
In a symbol of the form $\langle a_1, \ldots, a_i; b_1, \ldots, b_j;c_1, \ldots, c_k\rangle$,  the numbers $i$, $j$, and $k$  are the numbers of orbits of  vertices, faces and Petrie-polygons; the $a$'s, $b$'s and $c$'s denote the sizes of the vertices, faces and Petrie-polygons in each  particular orbit. We further note that the edge-multiplicity of all the maps in the table is 1 and all the maps are non-degenerate (i.e. all parameters in map symbol are greater or equal to 3). Column {\bf MType} denotes the type of $Me(\ma)$, while the column {\bf Type} denotes the type of $\ma$. In {\bf Genus} column we use special notation, namely non-negative numbers denote orientable genus while negative numbers denote non-orientable genus of both maps. 
\end{sidewaystable}

To conclude this section, we introduce a short discussion that deals with the medial of a medial map.

\begin{proposition}
Let $\ma$ be a $k$-orbit map. If $Me(Me(\ma))$ is also a $k$-orbit map, then $\ma$ has Schl\"afli type $\{4,4\}$.
\end{proposition}

\begin{proof}
$Me(Me(\ma)$ is a $k$-orbit map if and only if both $Me(\ma)$ and $\ma$ are self-dual maps. Since $Me(\ma)$ is a medial map, then each of its vertices has valency 4; the fact that is self-dual implies that $Me(\ma)$ has Schl\"afly type $\{4,4\}$. On the other hand the faces of $Me(\ma)$ correspond to the vertices and faces of $\ma$. Because each face of $Me(\ma)$ is a 4-gon, each face of $\ma$ is also a 4-gon and each vertex of $\ma$ has valency 4, implying the proposition.
\end{proof}

The maps of type $\{4,4\}$ are maps on the torus or on the Klein Bottle. 
In \cite{toroids}, Hubard, Orbani\'c, Pellicer and Weiss study the symmetry types of equivelar maps on the torus. The maps of type $\{4,4\}$ on the torus have symmetry type 1, 2, $2_1$, $2_{02}$ or $4_{C_p}$ and are all self-dual. The medial of a map $\{4,4\}$ on the torus of type 1, 2 or $4_{C_p}$ is of the same type as the original, while for types $2_1$ and $2_{02}$ the medial is precisely of the other type. Therefore $Me(Me(\ma))$ has the same symmetry type graph, whenever $\ma$ is a map on the torus of Schl\"afli type $\{4,4\}$.

In \cite{KleinBottle} Wilson shows that there are two kinds of map of type $\{4,4\}$ in the Klein Bottle, and denotes them by $\{4,4\}_{\setminus m,n \setminus}$ and $\{4,4\}_{|m,n|}$, respectively. From the \cite[Table I]{KleinBottle} we can see that these maps have $2mn$ edges, and thereby $8mn$ flags.
Moreover, the automorphism group of $\{4,4\}_{\setminus m,n \setminus}$ it has $4m$ elements,
while for $\{4,4\}_{|m,n|}$ has $8m$ elements if $n$ is even and $4m$ otherwise.
Thus, $\{4,4\}_{\setminus m,n \setminus}$ is a $2n$-orbit map and $\{4,4\}_{|m,n|}$ has $n$ flag orbits if $n$ is even and $2n$ otherwise.

For a map of type $\{4,4\}_{\setminus m,n \setminus}$, it can be seen that
$Me(Me(\{4,4\}_{\setminus m,n \setminus})) = \{4,4\}_{\setminus 2m,2n \setminus}$;
which has $32mn$ flags and its automorphism group has $8m$ elements.
Hence, the map $Me(Me(\{4,4\}_{\setminus m,n \setminus}))$ is a $4n$-orbit map (i.e. has two times the number of orbits than the map $\{4,4\}_{\setminus m,n \setminus}$).
On the other hand, if the map $\ma$ is of type $\{4,4\}_{|m,n|}$, the map $Me(Me(\ma))$ is the dual map of $\{4,4\}_{|2m,2n|}$. 
Since for any map and its dual have the same number of flag orbits, and the edges on both maps are in one-to-one correspondence, we can compute that $Me(Me(\ma))$ has $32mn$ flags and its automorphism group has $16m$ elements. Hence the map $Me(Me(\{4,4\}_{|m,n|}))$ is a $2n$-orbit map. We therefore have the following proposition.

\begin{proposition}
Let $\ma$ be a $k$-orbit map. Then $Me(Me(\ma))$ is a $k$-orbit map if $\ma$ is a map on the torus of type $\{4,4\}$, or is a map on the Klein Bottle of type $\{4,4\}_{|m,n|}$, where $n$ is odd.
\end{proposition}

\section{Conclusion}

We have presented a method that helps enumerating medial type graphs. The results are presented in Table \ref{table}. 
The first row, $(a)$, of this Table gives, for each value of $1 \leq k \leq 10$ (number of orbits on a map), the number of all possible symmetry type graphs with $k$-vertices.

In row $(b)$ we say how many of the symmetry type graphs with $k$ vertices are self-dual, while  row $(c)$ given the number of them that have polarities. Thereby $(c)$ counts the number of self-dual symmetry type graphs with $k$-vetices that might be obtained from  self-dual $k$-orbit maps. Note that for $k=8$ there is a self-dual symmetry type graph with no polarities (see Figure~\ref{nopolarity}).

\begin{figure}[htbp]
\begin{center}
\includegraphics[width=6cm]{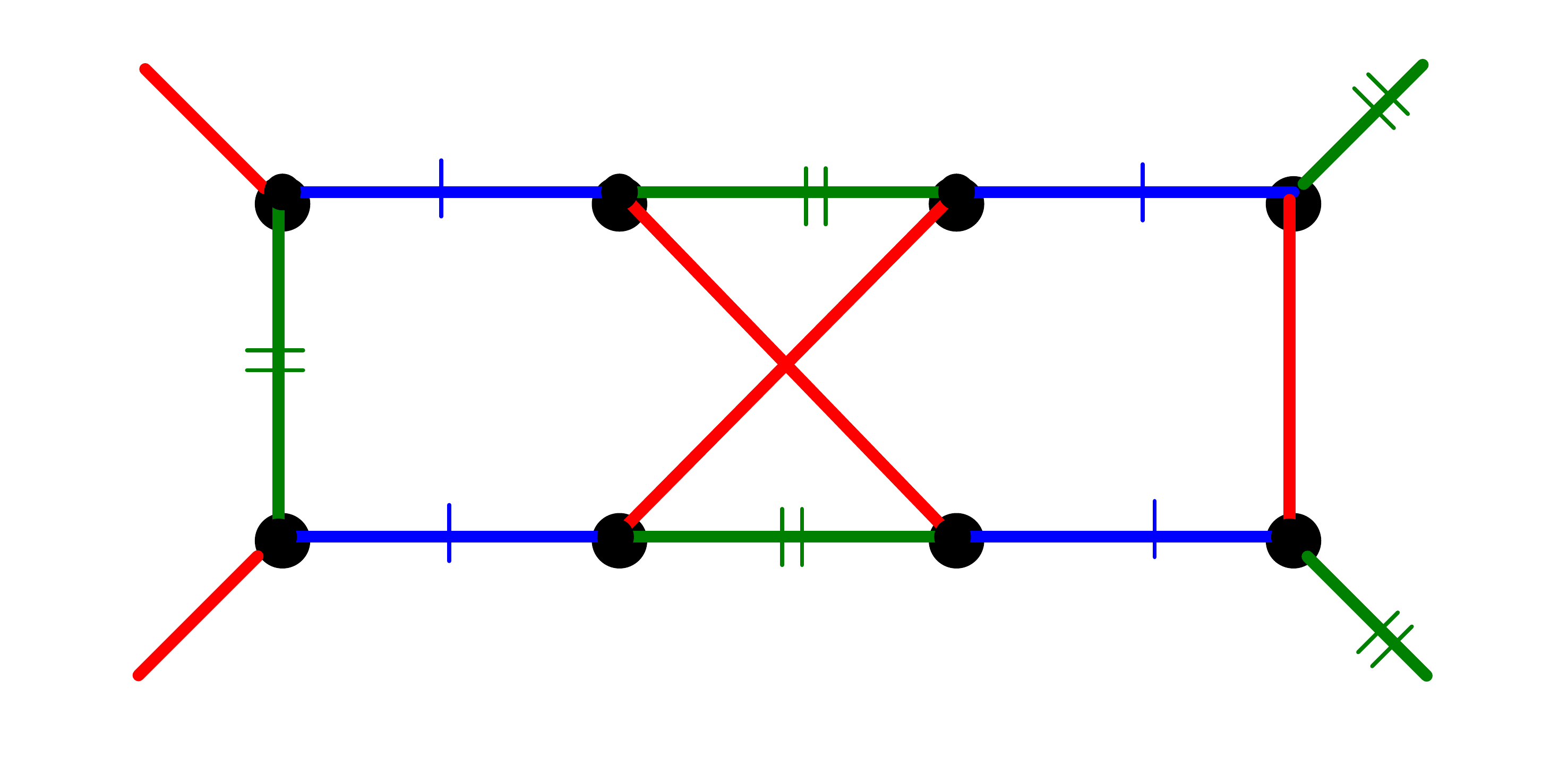}
\caption{A self-dual symmetry type graph with 8 vertices and no polarities}
\label{nopolarity}
\end{center}
\end{figure}

As we saw before, each self-dual symmetry type graph might have more than one duality (for example, it might have a proper and an improper duality). Row $(d)$ gives the number of dualitites that all self-dual symmetry type graphs with $k$ vertices have. Row $(e)$ tells us how many of them are polarities. (Again, these last ones are the ones that can arise from dualities of self-dual maps.) In other words, row $(e)$ tells us the number of extended symmetry type graphs with $k$ vertices.

Each extended symmetry type graph with $k$ vertices gives raise to a medial symmetry type graph with $k$ vertices. However, two different such extended graph may give raise to the same medial symmetry type graph. Row $(f)$ gives the number of medial symmetry type graphs with $k$ vertices that arise from extended symmetry type graphs with $k$ vertices, while row $(g)$ gives the total number of medial symmetry type graphs with $k$ vertices. 
We observe that for $1 \leq n \leq 10$, 
if we denote by $g_n$ the number in the cell $(g)$ 
corresponding to $(n+1)^{th}$ column, $g_{2k}$ can be computed in the following way.
$$g_{2k} = f_{2k} + \frac{b_k+a_k}{2}, \;\;\;\;\; g_{2k-1} = f_{2k-1}$$
where $f_n$, $b_{n}$ and $a_{n}$ are the respective values given on the cells $(f)$, $(b)$ and $(a)$ in the column $n$. We conjecture that this is the case for any integer  $n$.

\begin{table}[h!]
\centering
\footnotesize{
\begin{tabular}{|c|c|c|c|c|c|c|c|c|c|c|c|}
\hline
   $k$         & 1 & 2 & 3 & 4  & 5  & 6  & 7  &  8  &  9  & 10   &       \\
\hline
No. of types   & 1 & 7 & 3 & 22 & 13 & 70 & 67 & 315 & 393 & 1577 & $(a)$ \\
\hline
No. of self    & 1 & 3 & 1 & 8  & 3  & 12 & 7  & 45  & 25  & 91   & $(b)$ \\
 dual types    &   &   &   &    &    &    &    &     &     &      &       \\
\hline
No. of self    & 1 & 3 & 1 & 8  & 3  & 12 & 7  & 44  & 25  & 91   & $(c)$ \\
polar types    &   &   &   &    &    &    &    &     &     &      &       \\
\hline
No. of self    & 1 & 6 & 1 & 21 & 3  & 23 & 7  & 101 & 25  & 128  & $(d)$ \\
dualities      &   &   &   &    &    &    &    &     &     &      &       \\
\hline
No. of self    & 1 & 6 & 1 & 17 & 3  & 21 & 7  & 83  & 25  & 124  & $(e)$ \\
polarities     &   &   &   &    &    &    &    &     &     &      &       \\
\hline
No. of  medial& 1 & 6 & 1 & 15 & 3  & 19 & 7  & 73  & 25  & 120  & $(f)$ \\
  types from $k$-orb maps  &   &   &   &    &    &    &    &     &     &      &       \\
\hline
No. of total   & 1 & 7 & 1 & 20 & 3  & 21 & 7  & 88  & 25  & 128  & $(g)$ \\
medial types    &   &   &   &    &    &    &    &     &     &      &       \\
\hline
 \end{tabular}}
\caption{Number of symmetry type graphs, self-dual types and medials types}
\label{table}
\end{table}

\bigskip
\noindent\textit{Acknowledgments.}
The authors would like to thank Gunnar Brinkmann
and Nico Van Cleemput for computer check of our numbers in Table \ref{table}.
This work has been financed by ARRS within the EUROCORES Programme EUROGIGA (project GReGAS, N1--0011) 
of the European Science Foundation.
The research of Hubard was supported by PAPIIT-M\'exico under project IB101412. 
The research of other authors was supported in part by ARRS,  Grant P1-0294.

\end{document}